\documentclass[12pt]{article}
\usepackage{listings}
\lstdefinestyle{mystyle}{
    showtabs=false,                  
    tabsize=3
}

\usepackage{currfile,datetime}
\usepackage{amsthm,amsmath,amssymb,bbm,geometry,epsfig,listings,
paralist,enumerate,comment,nicefrac,verbatim,multirow,xcolor,wasysym,tikz}
\usepackage{caption}
\usepackage{subcaption}
\usepackage[pagewise]{lineno}[4.41]
\usepackage{mathrsfs}  
\usepackage{marginnote}\reversemarginpar
\usepackage[linktocpage,colorlinks,hidelinks,linkcolor=red,citecolor=green]{hyperref}
\usepackage{soul}

\usepackage[capitalise,compress]{cleveref} 

\crefname{equation}{}{}

\newtheorem{lemma}{Lemma}[section]

\newtheorem{theorem}[lemma]{Theorem}

\newtheorem{corollary}[lemma]{Corollary}

\theoremstyle{definition}
\newtheorem{example}[lemma]{Example}

\crefname{subsection}{Subsection}{Subsections}
\crefname{enumi}{item}{items}
\creflabelformat{enumi}{(#2\textup{#1}#3)}

\newcommand{\1}{\ensuremath{\mathbbm{1}}}
\providecommand{\N}{{\ensuremath{\mathbbm{N}}}}
\providecommand{\Z}{{\ensuremath{\mathbbm{Z}}}}
\providecommand{\R}{{\ensuremath{\mathbbm{R}}}}

\providecommand{\E}{{\ensuremath{\mathbbm{E}}}}

\renewcommand{\P}{{\ensuremath{\mathbbm{P}}}}
\renewcommand{\gets}{\curvearrowleft}
\newcommand{\F}{{\ensuremath{\mathbbm{F}}}}

\newcommand{\procX}{X}
\newcommand{\setS}{\mathbbm{S}}

\newcommand{\funcPi}{\iota}
\newcommand{\setSB}{\widetilde{\mathbbm{S}}}
\newcommand{\rdown}[2]{#2(#1)}
\newcommand{\size}[1]{\lvert #1 \rvert}

\newcommand{\totalD}{\mathsf{D}}

\newcommand{\exponentLP}{p}

\newcommand{\tnorm}[2]{{\left\vert\kern-0.25ex\left\vert\kern-0.25ex\left\vert #1     \right\vert\kern-0.25ex\right\vert\kern-0.25ex\right\vert}_{#2}}

\newcommand{\barC}{\bar{c}}

\newcommand{\xeqref}[1]{}

\title{Strong convergence rate of Euler--Maruyama approximations in temporal-spatial H\"older-norms}

\author
{Martin Hutzenthaler$^{1}$ \\
 Tuan Anh Nguyen$^{2}$\bigskip\\
\small{$^1$ Faculty of Mathematics, University of Duisburg-Essen,}\\
\small{Essen, Germany; e-mail: \texttt{martin.hutzenthaler}\textcircled{\texttt{a}}\texttt{uni-due.de}}\\
\small{$^2$ Faculty of Mathematics, University of Duisburg-Essen,}\\
\small{Essen, Germany; e-mail: \texttt{tuan.nguyen}\textcircled{\texttt{a}}\texttt{uni-due.de}}
}

\AtBeginDocument{%
\sloppy
\allowdisplaybreaks
}

\begin{document}

\maketitle
\begin{abstract}
 Classical approximation results for stochastic differential equations
 analyze the $L^p$-distance between the exact solution and its Euler--Maruyama approximations.
 In this article we measure the error with temporal-spatial H\"older-norms.
 Our motivation for this are multigrid approximations of the exact solution
 viewed as a function of the starting point.
 We establish the classical strong convergence rate $0.5$ with respect to
 temporal-spatial H\"older-norms
 if the coefficient functions have bounded derivatives of first and second order.
\end{abstract}
\tableofcontents
{%
\makeatletter
\let\@makefnmark\relax
\let\@thefnmark\relax
\@footnotetext{\emph{Key words and phrases:}
stochastic differential equation, strong convergence, Euler--Maruyama approximation,  Lipschitz condition, Lyapunov function,
curse of dimensionality, high-dimensional SDEs, 
high-dimensional BSDEs, multilevel Picard approximations, multilevel Monte Carlo method. }
\@footnotetext{\emph{AMS 2010 subject classification}:Primary 60H35; Secondary 65C05, 65C30.} 
\makeatother
}%

\section{Introduction}
Stochastic differential equations are typically not explicitly solvable
and need to be approximated numerically.
Classical results on strong convergence rates assume the SDE coefficients to
be globally Lipschitz continuous; see, e.g., \cite{KP95,Mao07}.
In the last decade, strong convergence rates were also established
in the case of non-globally Lipschitz coefficients.
We refer, e.g., to
 \cite{HutzenthalerJentzen2014Memoires,HJ20,HutzenthalerJentzenKloeden2012,Sabanis2013ECP,Sabanis2016}
for the case of locally Lipschitz coefficients with polynomial growth
where Euler approximations diverge in the strong and weak sense; see \cite{hjk11,HutzenthalerJentzenKloeden2013}.
Moreover, we refer, e.g., to 
\cite{BaoHuangZhang2020,DareiotisGerencserLe2021,LeLing2021,LeobacherSzoelgyenyi2016,NeuenkirchSzoelgyenyi2021,NeuenkirchEtAl2019,NgoTaguchi2017}
for the case of discontinuous drift coefficients.
The error is measured in all of these approximation results as $L^p$-distance between the
exact solution and the approximations for at least one $p\in[1,\infty)$.

In this article we measure the error with temporal-spatial H\"older-norms.
Our motivation for considering spatial H\"older-norms
is that in a number of applications we need to approximate the solution in several
starting points, e.g., in a domain or a submanifold.
It is inefficient to approximate the exact solution (as a function of the starting point)
by interpolating over a fine subgrid. More efficient is to apply multigrid approximations
and to exploit the spatial regularity of the approximation processes.
To give a specific example, the (first component of the) solution of a backward stochastic
differential equation (BSDE) can be written as $(u(t,X_t))_{t\in[0,1]}$ where $X$ is the forward process
 and where $u$ solves a backward partial differential equation (PDE).
Assuming that the PDE is linear, we can approximate $u$ by Monte Carlo Euler approximations
$(U_n)_{n\in\N}$. It is inefficient to approximate $(u(t,X_t))_{t\in[0,1]}$ 
by an affine-linear interpolation of $(U^n(k2^{-n},X_{k2^{-n}}))_{k\in\{0,1,\ldots,2^n\}}$, $n\in\N$.
More efficient than this are multigrid approximations
\begin{equation}  \begin{split}
  \sum_{\ell=0}^{n-1}\Bigl[
  \mathscr{L}\big((U^{n-\ell}(k2^{-\ell-1},X_{k2^{-\ell-1}}))_{k\in\{0,1,\ldots,2^{\ell+1}\}}\big)
  -\1_{\N}(\ell)\mathscr{L}\big((U^{n-\ell}(k2^{-\ell},X_{k2^{-\ell}}))_{k\in\{0,1,\ldots,2^\ell\}}\big)
  \Bigr]
\end{split}     \end{equation}
where $\1_{\N}$ is the indicator function of $\N=\{1,2,\ldots\}$ and
for every
$N\in \N$, $a\colon  \{0,1,,\ldots,N\}\to \R$
we denote by
$\mathscr{L}(a)\in C([0,1],\R)$ the continuous function which satisfies for all
 $k\in\{0,1,\ldots,N-1\} $, $t\in [\frac{k}{N},\frac{k+1}{N}]$ that
$
(\mathscr{L}(a))(t)=a_k(k+1-Nt)+a_{k+1}(Nt-k)
$;
 see \cite[Theorem 2.3]{hjkn2021BSDE}
for more details and cf.\ also \cite{h98,h01}.
For the analysis of these multigrid approximations, however, we need to understand the temporal-spatial
regularity of Monte Carlo Euler approximations.
If the PDE is nonlinear and high-dimensional, then we can approximate its solution
by multilevel Picard approximations;
see, e.g., \cite{EHutzenthalerJentzenKruse2016,HJK+18,hjk2021overcoming}.
Also in this case we need to understand regularity of Euler--Maruyama approximations considered as functions of the starting point.

The following \cref{a01} illustrates the main result of this article and proves
that Euler--Maruyama approximations converge with strong convergence rate $0.5$ also in spatial
$1$-H\"older norms under suitable assumptions. The central assumption of \cref{a01} is that
the coefficient functions $\mu$ and $\sigma$ are twice continuously differentiable and that
the derivatives of first and second order are bounded.
In fact, it suffices to assume a weaker assumption on the coefficient functions, namely
 there exists $c\in\R$ such that for all $\zeta\in\{\mu,\sigma\}$, $x,y\in\R^d$ it holds that
\begin{equation} \begin{split}\label{eq:second.order}
&
\left\lVert 
(\zeta(x)-\zeta (y))-
(\zeta(\tilde{x})-\zeta (\tilde{y}))\right\rVert 
\leq c
\left\lVert 
(x-y)-(\tilde{x}-\tilde{y})\right\rVert 
+c
\left(\left\lVert x-y\right\rVert 
+\left\lVert \tilde{x}-\tilde{y}\right\rVert 
\right)\lVert x-\tilde{x}\rVert,
\end{split} \end{equation}
cf.\ \cref{s01b} for details.
So global Lipschitz continuity of the coefficient functions is sufficient
for Euler--Maruyama approximations to converge
with strong convergence rate $0.5$ whereas condition \eqref{eq:second.order} is sufficient
to obtain strong convergence rate $0.5$ with respect to spatial H\"older norms.
Clearly, \eqref{eq:second.order} implies global Lipschitz continuity of $\mu$ and $\sigma$.
We believe, however, that condition \eqref{eq:second.order} is not necessary for
\eqref{eq:hoelder.intro} to hold.
It would be interesting to prove convergence rates in strong spatial H\"older norms
in the situations considered in the articles mentioned in the first paragraph, that is,
when the coefficients are only locally Lipschitz continuous or when the
drift coefficient is only measurable (plus suitable additional assumptions).

\begin{theorem}\label{a01}
Let $d\in\N$, 
$T,p\in (0,\infty)$,
let $\lVert\cdot\rVert \colon \R^d \to [0,\infty) $ be a norm,
let 
$\mu\in C^2(\R^d,\R^d) $, $ \sigma  \in C^2(\R^d,\R^{d\times d})$ 
have bounded first and second order derivatives,
let $(\Omega,\mathcal{F},\P, (\F_t)_{t\in[0,T]})$ be a filtered probability space which satisfies the usual conditions\footnote{Let $T \in (0,\infty)$ and let ${\bf \Omega} = (\Omega,\mathcal{F},\P, (\F_t)_{t\in[0,T]})$ be 
a filtered probability space. 
Then we say that ${\bf \Omega}$
satisfies the usual conditions if and only if 
it holds for all $t \in [0,T)$ that $\{ A\in \mathcal F: \P(A)=0 \} \subseteq \F_t 
= \cap_{ s \in (t,T] } \F_s$.},
let $(W_t)_{t\in[0,T]}\colon [0,T]\times\Omega\to\R^{d}$ be a standard
$(\F_t)_{t\in[0,T]}$-Brownian motion with continuous sample paths, and for every $n\in \N$, $x\in \R^d$ let 
${Y}_{k}^{n,x}\colon\Omega \to \R^d $, $k\in \{0,\ldots,n\}$,
satisfy for all $k\in  \{0,\ldots,n-1\}$ that
${Y}_{0}^{n,x}=x$ and 
$
{Y}_{k+1}^{n,x}= {Y}_{k}^{n,x} + 
\mu({Y}_{k}^{n,x}) \tfrac{T}{n}
+\sigma({Y}_{k}^{n,x})( W_ {\frac{(k+1)T}{n}}-
W_ {\frac{kT}{n}}).
$
Then
\begin{enumerate}[i)]
\item\label{a01a} 
for every $x\in\R^d$ there exists a unique adapted stochastic process with continuous sample paths
 $(X^{x}_t)_{t\in[0,T]} \colon [0,T]\times\Omega\to\R^d $  
  such that
for all $t\in[0,T]$ it holds a.s.\ that $X_t^{x}=x+\int_{0}^{t}\mu(X_s^{x})\,ds+\int_{0}^{t}\sigma(X_s^{x})\,dW_s$ and
\item\label{a01b} there exists $C\in[0,\infty)$ which satisfies  
for all $n\in \N$,
$k\in \{0,1,\ldots,n\}$, $x\in\R^d$, $y\in \R^d\setminus \{x\}$ that
\begin{align}\begin{split}\label{eq:hoelder.intro}
\frac{ \left(\E \!\left[\bigl\lVert X_{\frac{kT}{n}}^x- {Y}_{k}^{n,x}\bigr\rVert^p\right]\right)^{1/p}}{1+\lVert x\rVert}
+
\frac
{ \left(\E \!\left[\bigl\lVert 
\bigl( X_{\frac{kT}{n}}^x- {Y}_{k}^{n,x}\bigr)
-\bigl( X_{\frac{kT}{n}}^y- {Y}_{k}^{n,y}\bigr)
\bigr\rVert^p\right]\right)^{1/p}}{\lVert x-y\rVert\bigl(1+\lVert x\rVert+\lVert y\rVert\bigr)}\leq \frac{C}{\sqrt{n}}.
\end{split}\end{align}
\end{enumerate}
\end{theorem}
\cref{a01}  follows from
 \cref{c10}, H\"older's inequality, and the fact that for all $d\in\N$ it holds that all norms on $\R^d$ are equivalent.
In order to illustrate \cref{a01} by a
numerical experiment
we consider an example where the SDE solution is explicit,  \cref{e01} below.
\begin{example}\label{e01}
Consider the setting of \cref{a01}, assume for all $x\in\R$ that
$T=1$, $d=1$, $p=2$, $\mu(x)= -\sin(x)\cos^3(x) $, and $\sigma(x)= \cos^2(x)$ and for every $n\in \N$ let 
$\Delta_n\in\R  $ 
satisfy that
$\Delta_n=
 \bigl(\E \bigl[\bigl\lvert 
( X_{1}^1- {Y}_{n}^{n,1})
-( X_{1}^{1+\frac{1}{n}}- {Y}_{n}^{n,1+\frac{1}{n}})
\bigr\rvert^2\bigr]\bigr)^{1/2}.
$ We approximate each expectation by 
a Monte-Carlo average over $1000$ independent samples.
Note that for all
$t\in[0,1]$, $x\in\R$ it holds a.s.\
that \begin{align}
X_t^x=x-\int_{0}^{t}\sin(X_s^x)\cos^3(X_s^x)\,ds+\int_{0}^{t}\cos^2(X^x_s)\,dW_s
\end{align}
and $X_t^x= \arctan(aW_t+\tan(x))$ (see, e.g., \cite[(4.27) in p. 121]{KP95}).
  \cref{a01} shows that $\sup_{n\in\N}[\Delta_nn^{-1.5}]<\infty$. 
Our simulations support this result and show that
$ (2^n,\Delta_{2^n})_{n\in \{1,2,\ldots,10\}} $, decays  like 
 $ (2^n,(2^n)^{-1.5}) _{n\in \{1,2,\ldots,10\}} $; see \cref{f10}.
\end{example}

\begin{figure}\centering
\includegraphics[scale=1]{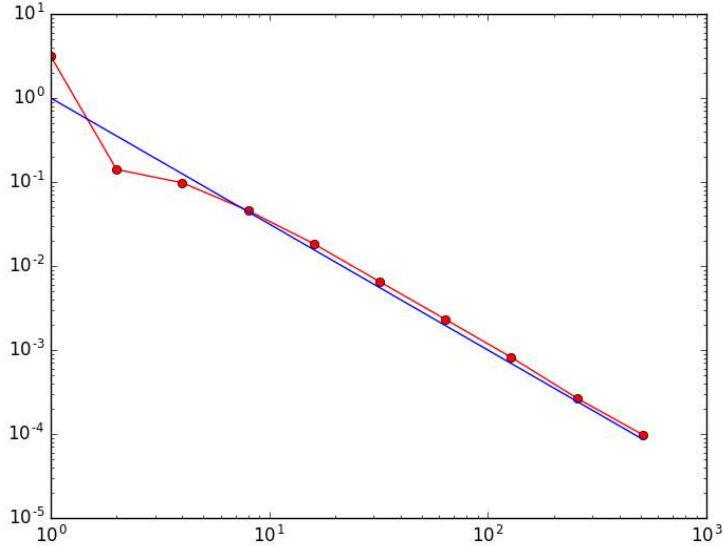}\caption{Simulation by Python using numpy and matplotlib.pyplot for \cref{e01}. The  line with dots contains the points $ (2^n,\Delta_{2^n}) $ for $n\in \{1,2,\ldots,10\} $ in \cref{e01}. The straight line is the reference line $(2^n,(2^n)^{-1.5})_{n\in\{1,2,\ldots,10\}}. $}\label{f10}
\end{figure}

The remainder of this article is organized as follows.
In \cref{sec:Gronwall} we generalize Gronwall's lemma to the case of general temporal
discretizations. 
In \cref{s01} we prove under suitable assumptions
that Euler--Maruyama approximations converge in temporal-spatial H\"older $L^p$-norms
with rate $0.5$ (cf.\ \cref{s01b} and \cref{c10}). As a corollary hereof, we prove in  \cref{c01} that Monte Carlo Euler approximations
converge in spatial Lipschitz $L^p$-norms with rate $0.5$.

\section{Gronwall inequalities}\label{sec:Gronwall}
\cref{b01} generalizes the well-known discrete and continuous Gronwall lemma.
\begin{lemma}\label{b01}
Let $a,c\in [0,\infty)$,
$T\in \R$, $t_0\in (-\infty,T)$, let
 $x\colon [t_0,T]\to [0,\infty) $, $\delta \colon [t_0,T]\to [t_0,T]$ be
 measurable, and assume for all $t\in[t_0,T]$ that 
$\delta(t)\leq t$ and 
 $
x(t)\leq a+\int_{t_0}^t c x(\delta(s))\,ds<\infty.
$
Then for all $t\in [t_0,T]$ it holds that $x(t)\leq a e^{c(t-t_0)}$.
\end{lemma}
\begin{proof}[Proof of \cref{b01}]
The assumptions on $x$ and $\delta$
 imply for all $t\in[t_0,T]$ that
$
x(\delta(t)) \leq a+ \int_{t_0}^{\delta(t)}c x(\delta(s))\,ds
\leq a + \int_{t_0}^{t}c x(\delta(s))\,ds<\infty
$. This and Gronwall's lemma (see, e.g., \cite[Lemma~2.11]{GrohsWurstemberger2018})
 show for all
$t\in[t_0,T]$ that $x(\delta(t))\leq a e^{c(t-t_0)}$.
Therefore, the assumptions on $x$ and $\delta$  imply for all
$t\in[t_0,T]$
that 
$
x(t)\leq  a+\int_{t_0}^t c x(\delta(s))\,ds\leq 
a (1+c\int_{t_0}^t e^{c(s-t_0)}\,ds)
= a e^{c(t-t_0)}
.
$
This completes the proof of \cref{b01}.
\end{proof}

\begin{corollary}\label{b02}
Let $c\in [0,\infty)$, $p\in [1,\infty)$,
$T\in \R$, $t_0\in (-\infty,T)$, let
 $a,x\colon [t_0,T]\to [0,\infty) $, $\delta \colon [t_0,T]\to [t_0,T]$ be
 measurable, and assume for all $t\in[t_0,T]$ that 
$\delta(t)\leq t$
 and $
x(t) \leq a(t)+\bigl(\int_{t_0}^t \lvert c x(\delta(s))\rvert^p\,ds\bigr)^{\nicefrac{1}{p}}<\infty.
$
Then 
for all $t\in[t_0,T]$ it holds that 
\begin{align}\textstyle
 x(t) \leq  2^{1-\frac{1}{p}} \left[\sup_{s\in [t_0,t]}a(s)\right] \exp (\tfrac{2^{p-1}c^p (t-t_0)}{p}  ).
\end{align}
\end{corollary}
\begin{proof}[Proof of \cref{b02}]
The assumptions and the fact that
$\forall \, A,B\in[0,\infty)\colon (A+B)^p\leq 2^{p-1}(A^p+B^p)$
show for all  ${t}\in[t_0,T]$, $\bar{t}\in[t,T]$ that 
$
\lvert x(t)\rvert^p  \leq 2^{p-1}\sup_{s\in[t_0,\bar{t}]}a^p(s)+2^{p-1}\int_{t_0}^t c^p\lvert x(\delta(s))\rvert^p\,ds<\infty.
$ This and \cref{b01} (applied for every $\bar{t}\in(t_0,T]$ with
$a\gets 2^{p-1}\sup_{s\in[t_0,\bar{t}]}a^p(s) $, $c\gets 2^{p-1}c^p$,
$T\gets \bar{t} $,
$x\gets x^p|_{[t_0,\bar{t}]} $ in the notation of \cref{b01}\footnote{Here and throughout the paper $a \gets b$ should be read as ``$a$ replaced by $b$''.})
imply for all ${t}\in(t_0,T]$, $\bar{t}\in[t,T]$ that 
$\lvert x(t)\rvert^p\leq  2^{p-1}\sup_{s\in [t_0,\bar{t}]}a^p(s) \exp (2^{p-1}c^p (t-t_0)  )$.
This and the fact that
$x(t_0)\leq a(t_0)$
 complete the proof of \cref{b02}.
\end{proof}

\section{Strong convergence rate of Euler--Maruyama approximations in temporal-spatial H\"older norms}
\label{s01}

In \cref{s01} we prove under suitable assumptions that Euler--Maruyama approximations converge in temporal-spatial H\"older $L^p$-norms
with rate $0.5$ (cf.\ \cref{s01b} and \cref{c10}). As a corollary hereof, we prove in  \cref{c01} that Monte Carlo Euler approximations
converge in spatial Lipschitz $L^p$-norms with rate $0.5$. 
The central assumption for this is \eqref{v03}.
In \cref{f01} below we provide the well-known fact that a $C^2$-function with bounded first and second order derivatives
satisfies this condition \eqref{v03}. First, we provide in \cref{s28} well-known upper bounds for 
polynomials
 which we
use as Lyapunov-type functions.


\begin{lemma}\label{s28}
 Let $d\in\N$, 
let $\lVert \cdot \rVert\colon \R^d\to [0,\infty)$ be the standard norm, and let
$p\in [3/2,\infty)$, $a,c\in[0,\infty)$,
$V \colon \R^d\to \R $ 
 satisfy for all $x\in\R^d$ that
$
V (x)= ({a} +c^2\lVert x\rVert  ^2 )^p.
$
Then 
it holds for all $x,y,z\in\R^d$ that
$
\bigl\lvert((\totalD V)(x))(y)\bigr\rvert
\leq 2pc (V(x))^{\frac{2p-1}{2p}}\lVert y\rVert$ and
$
\bigl\lvert ((\totalD^2 V)(x))(y,z)\bigr\rvert
\leq 2p(2p-1)c^2(V (x))^{\frac{p-1}{p}}\lVert y\rVert \lVert z\rVert$.
\end{lemma}
\begin{proof}[Proof of \cref{s28}]Throughout the proof let 
$\langle\cdot ,\cdot \rangle\colon \R^d\times\R^d\to\R $ be the standard scalar product. The assumptions and the Cauchy-Schwarz inequality
 show   for all  $x,y,z\in\R^d$  that
 $c \lVert x\rVert  \leq (V (x))^{\frac{1}{2p}}$, 
$ \bigl\lvert
[(\totalD V )(x)](y)\bigr\rvert= 
\bigl\lvert
p\bigl[{a} +c^2\lVert x\rVert  ^2 \bigr]^{p-1}2c^2\langle x,y\rangle\bigr\rvert\leq 
2pc^2(V  (x))^{\frac{p-1}{p}}\lVert x\rVert\lVert y\rVert
\leq 
2pc(V(x))^{\frac{2p-1}{2p}}\lVert y\rVert,
$ and
\begin{align}
\begin{split}
\left\lvert [(\totalD^2 V )(x)](y,z)\right\rvert
&= \left\lvert p(p-1)\left[{a} +c^2\lVert x\rVert  ^2 \right]^{p-2}2c^2\langle x,z\rangle
2c^2
\langle x,y\rangle+2pc^2\left[{a} +c^2\lVert x\rVert  ^2 \right]^{p-1}\langle y,z\rangle\right\rvert
\\
&\leq  4p(p-1)c^4(V (x))^{\frac{p-2}{p}}\lVert x\rVert^2\lVert y\rVert \lVert z\rVert+2pc^2(V  (x))^{\frac{p-1}{p}}\lVert y\rVert\lVert z\rVert\\
&\leq  4p(p-1)c^2(V (x))^{\frac{p-1}{p}}\lVert y\rVert \lVert z\rVert+2pc^2(V  (x))^{\frac{p-1}{p}}\lVert y\rVert\lVert z\rVert\\
&\leq (4p^2-2p)c^2(V (x))^{\frac{p-1}{p}}\lVert y\rVert \lVert z\rVert.
\end{split}
\end{align}
This 
completes the proof of \cref{s28}.
\end{proof}

 The following theorem, \cref{s01b}, is the main result of this article and establishes strong convergence
  rate $0.5$ of Euler--Maruyama approximations in temporal-spatial H\"older norms.
The estimates in \cref{s01b} are explicit in all parameters.
  In particular, this allows to identify situations where the approximation error grows at most polynomially in the dimension
  and where the Euler--Maruyama approximations do not suffer from the curse of dimensionality.
Note that $b=\infty$ is the case of globally Lipschitz continuous coefficients and in this case the right-hand sides
 of \eqref{m33b} and \eqref{m32b} are trivial (in $\{0,\infty\}$). 
Moreover, observe that \eqref{y01} and the fact that $\forall\,t\in [0,T]\colon\funcPi(t)=t $ imply
for every 
 $s\in [0,T]$, $x\in\R^d$  that
 $(X^{\funcPi,x}_{s,t})_{t\in[s,T]} $
is the exact solution to the SDE with coefficient functions $\mu,\sigma$ starting at $(s,x)$.
Furthermore,
\eqref{y01} implies that
for all 
$\delta\in \setS$,
$n\in\N$,
$
t_0,t_1,\ldots,t_{n}\in [0,T]$, $k\in[0,n-1]\cap\Z $,
$s\in[0,T]$, $t\in [s,T]$
 with  $0=t_0< t_1<\ldots<t_n=T$,
$ 
\delta ([t_0,t_1])=\{t_0\},\delta ((t_1,t_2])=\{t_1\},\ldots,
\delta ((t_{n-1},t_n])=\{t_{n-1}\}$, 
$t\in(t_k,t_{k+1}]$ it holds that $\delta(t)=t_k$, $X_{s,s}^{\delta,x}=x$, and
\begin{align}
 X^{\delta,x}_{s,t}
=X^{\delta,x}_{s,\max\{s,t_k\}} 
+\mu(X^{\delta,x}_{s,\max\{s,t_k\}} )(t-\max\{s,t_k\})
+\sigma(X^{\delta,x}_{s,\max\{s,t_k\}} )(W_t-W_{\max\{s,t_k\}})
,
\end{align} that is,
$(X^{\delta,x}_{s,t})_{t\in[s,T]} $
is the Euler--Maruyama approximation to the SDE with coefficient functions $\mu,\sigma$ starting at $(s,x)$ and associated to the partition $(t_0,t_1,\ldots,t_n)$ of $[0,T]$.

Let us discuss the proof of \cref{s01b}.
First, 
 \eqref{r06b}--\eqref{m36} are standard results 
 and we include their proofs here for convenience and to have explicit constants. The main parts of \cref{s01b} are
 \eqref{m33}--\eqref{m32}.
While the estimate of a two point term is based on 
Gronwall's inequality and \eqref{k02},
the four point term in  \eqref{m33}
is estimated by using Gronwall's inequality and \eqref{m23}.
A crucial step for \eqref{m23} is \eqref{m12} in which
the regularity assumption of $\mu,\sigma$ in \eqref{v03}  is used to ``break'' the four point term containing $\mu,\sigma$ in \eqref{m23} into 
a four point term solely containing the SDE solution and its approximation.
To prove \eqref{m32}  we extend the four point estimate in \eqref{m33} to time regularity estimates. In
the proof of \eqref{m28}
 we have done similar things but for two point terms.


\begin{theorem}[Strong convergence of Euler--Maruyama approximations in H\"older norms]\label{s01b}
Let $\lVert \cdot\rVert\colon \bigcup_{k,\ell\in\N}\R^{k\times \ell}\to[0,\infty)$ satisfy for all $k,\ell\in\N$, $s=(s_{ij})_{i\in[1,k]\cap\N,j\in [1,\ell]\cap\N}\in\R^{k\times \ell}$ that
$\lVert s\rVert^2=\sum_{i=1}^{k}\sum_{j=1}^{\ell}\lvert s_{ij}\rvert^2$,
let $0\cdot \infty=0$,
let $d,m\in \N$, 
$T,c,\barC\in (0,\infty)$,
$b\in (0,\infty]$,
$p\in[2,\infty)$,
$\mu\in C(\R^d,\R^d)$, 
$\sigma\in C(\R^d,\R^{d\times m})$,  
$V\in C^2(\R^d,[1,\infty))$ satisfy for all
$x,y,\tilde{x},\tilde{y}\in\R^d$ that 
$
\lVert\mu(0)\rVert+\lVert\sigma(0)\rVert+c\lVert x\rVert\leq  (V(x))^{\nicefrac{1}{p}},
$
\begin{equation}
\bigl\lvert((\totalD V)(x))(y)\bigr\rvert
\leq \barC(V(x))^{\frac{p-1}{p}}\lVert y\rVert,\quad 
\bigl\lvert ((\totalD^2 V)(x))(y,y)\bigr\rvert
\leq \barC(V(x))^{\frac{p-2}{p}}\lVert  y\rVert^2,\label{m42}
\end{equation}
and
\begin{equation}
\begin{split}
&
\max _{\zeta\in\{\mu,\sigma\}}
\left\lVert 
(\zeta(x)-\zeta (y))-
(\zeta(\tilde{x})-\zeta (\tilde{y}))\right\rVert 
\\&\leq c
\left\lVert 
(x-y)-(\tilde{x}-\tilde{y})\right\rVert 
+b
\frac{\left\lVert x-y\right\rVert 
+\left\lVert \tilde{x}-\tilde{y}\right\rVert 
}{2}\lVert x-\tilde{x}\rVert,
\label{v03}
\end{split}
\end{equation}
let $\funcPi\colon [0,T] \to[0,T]$ satisfy for all $t\in[0,T]$ that
$\funcPi(t)=t$,
let $\setS$ satisfy that
\begin{equation}\small\begin{split}
\setS= \left \{
\delta\colon [0,T]\to [0,T]\colon 
\begin{aligned}
&\exists\,
n\in\N,
t_0,t_1,\ldots,t_{n}\in [0,T]\colon 0=t_0< t_1<\ldots<t_n=T, \\
&
\delta ([t_0,t_1])=\{t_0\},\delta ((t_1,t_2])=\{t_1\},\ldots,
\delta ((t_{n-1},t_n])=\{t_{n-1}\}
\end{aligned}
\right\},
\end{split}\end{equation}
let $\setSB=\setS\cup\{\funcPi\}$,
let $\size{\cdot}\colon \setSB\to [0,T]$ satisfy for all $\delta \in \setS$ that $\size{\funcPi}=0$ and
\begin{equation}
 \size{\delta} =\max
\Bigl\{\lvert s-t\rvert\colon s,t\in \delta([0,T]), s< t, (s,t)\cap \delta([0,T])=\emptyset  \Bigr\},
\end{equation}
let $(\Omega,\mathcal{F},\P, (\F_t)_{t\in[0,T]})$ be a filtered probability space which satisfies the usual conditions,
for every $s\in[1,\infty)$, 
$k,\ell\in\N$
and every random variable $\mathfrak{X}\colon \Omega\to\R^{k\times \ell}$ let $\lVert\mathfrak{X}\rVert_s\in[0,\infty]$ satisfy that $\lVert\mathfrak{X}\rVert_s^s=\E [\lVert\mathfrak{X}\rVert^s]$,
let $W=(W_t)_{t\in[0,T]}\colon [0,T]\times\Omega\to\R^{m}$ be a standard 
$(\F_t)_{t\in[0,T]}$-Brownian motion with continuous sample paths, and
for every 
$\delta\in \setSB$, $s\in [0,T]$, $x\in\R^d$ 
 let $(X^{\delta,x}_{s,t})_{t\in[s,T]} \colon[s,T]  \times\Omega\to\R^d $ 
be an adapted stochastic process with continuous sample paths such that
for all $t\in[s,T]$ it holds a.s.\ that
\begin{equation}
\label{y01}
 X^{\delta,x}_{s,t}=
x+\int_{s}^{t}\mu( X_{s,\max\{s,\rdown{r}{\delta}\}}^{\delta,x})\,d{r}
+\int_{s}^{t}\sigma( X_{s,\max\{s,\rdown{r}{\delta}\}}^{\delta,x})\,dW_{r}.
\end{equation}
Then
\begin{enumerate}[i)]
\item \label{r06b}it holds 
for all $\delta\in\setSB$,
 $s\in[0,T]$, $t\in [s,T]$,
 $x\in\R^d$
  that 
$
\E\bigl[ V(X_{s,t}^{\delta,x})\bigr]\leq e^{1.5\barC\lvert t-s\rvert}V(x)
,
$
\item \label{m29} it holds for all
 $\delta\in\setSB$,
$s\in[0,T]$, 
$t\in [s,T]$,
$x\in\R^d$ that
\begin{align}
\left\lVert 
X_{s,t}^{\delta,x}
-X_{s,t}^{\funcPi,x}\right\rVert_{p}
\leq 
 \sqrt{2}
c\left[
\sqrt{T}+p
\right]^2
e^{ c^2\left[
\sqrt{T}+p
\right]^2 T}
(e^{1.5\barC T}V(x))^{\nicefrac{1}{p} }
\lvert t-s\rvert^{\nicefrac{1}{2}} \size{\delta}^{\nicefrac{1}{2}}
,
\end{align}

\item 
\label{m28} it holds for all $\delta\in \setSB$,
$ s,\tilde{s}\in [0,T]$,
$t\in[s,T]$,
$\tilde{t}\in[\tilde{s},T]$,
 $x,\tilde{x}\in\R^d$ that
\begin{align}\begin{split}
&
\left \lVert X^{\delta,x}_{s,t}
-
X^{\delta,\tilde{x}}_{\tilde{s},\tilde{t}}\right\rVert_{p}\leq \sqrt{2}\lVert x-\tilde{x} \rVert 
e^{c^2\left[
\sqrt{T}+p
\right]^2 T}\\
&+
{5} e^{c^2\left[
\sqrt{T}+p
\right]^2T}
\left[
\sqrt{T}+p
\right]
e^{1.5 \barC T/p}
\frac{(V(x))^{\nicefrac{1}{p}}
+(V(\tilde{x}))^{\nicefrac{1}{p} }}{2}
\left[
\lvert s-\tilde{s}\rvert^{\nicefrac{1}{2}} +
\lvert t-\tilde{t}\rvert^{\nicefrac{1}{2}}\right]
,
\end{split}\end{align}

\item \label{m36}it holds
for all 
$\delta\in\setSB$, $s\in[0,T]$,
$t, \tilde{t}\in[s,T]$, 
$x,\tilde{x}\in\R^d$ that
\begin{align}
&\left \lVert 
\bigl(X^{\delta,x}_{s,\tilde{t}}-X_{s,t}^{\delta,x}\bigr)
-\bigl(X^{\delta,\tilde{x}}_{s,\tilde{t}}-X^{\delta,\tilde{x}}_{s,t}
\bigr)\right\rVert_{p}
\leq c\left[
\sqrt{T}+p
\right]\sqrt{2}
e^{c^2\left[
\sqrt{T}+p
\right]^2 T} \lVert x-\tilde{x}\rVert\lvert t-\tilde{t}\rvert^{\nicefrac{1}{2}}
,
\end{align}

\item 
\label{m33}
it holds for all 
$\delta\in\setSB$, $s\in[0,T]$,
$t\in[s,T]$, 
$x,\tilde{x},y,\tilde{y}\in\R^d$ that
\begin{align}
&\1_{ [4,\infty)} (p)\left\lVert 
\bigl(
X^{\funcPi,x}_{s,t}
-X^{\delta,y}_{s,t}
\bigr)
-\bigl(
X^{\funcPi,\tilde{x}}_{s,t}
-X^{\delta,\tilde{y}}_{s,t}
\bigr)
\right\rVert_{\frac{p}{2}}
  \leq \sqrt{2}e^{c^2\left[\sqrt{T}+p\right]^2T}
\left\lVert 
\bigl(
x-y \bigr)
-
\bigl(
\tilde{x}-\tilde{y} \bigr)\right\rVert\nonumber\\
&+
2\sqrt{2}(c^2+bc+b)
\left[\sqrt{T}+p\right]^4
e^{3 c^2\left[\sqrt{T}+p\right]^2 T}
e^{1.5\barC T /p}
\frac{(V(x))^{\nicefrac{1}{p}}
+(V(\tilde{x}))^{\nicefrac{1}{p} }}{2}
\size{\delta}^{\nicefrac{1}{2}}\lVert x-\tilde{x} \rVert
\lvert t-s\rvert^{\nicefrac{1}{2}}
\nonumber \\
& + 2\sqrt{2}b\left[\sqrt{T}+p\right] e^{3 c^2\left[
\sqrt{T}+p
\right]^2 T}
\frac{\left(\lVert x-y \rVert+
\lVert \tilde{x}-\tilde{y} \rVert\right) \lVert x-\tilde{x}\rVert}{2}
\lvert t-s\rvert^{\nicefrac{1}{2}},\label{m33b}
\end{align}
and
\item 
\label{m32}it holds
for all $\delta\in \setSB$,
 $s,\tilde{s}\in[0,T]$, 
$t\in [s,T]$, $ \tilde{t}\in[\tilde{s},T]$,
$x,\tilde{x}\in\R^d$ that
\begin{align}\begin{split}
&\1_{ [4,\infty)} (p)
\left\lVert \bigl(
X_{s,t}^{\funcPi,x}
-
X_{\tilde{s},\tilde{t}}^{\funcPi,\tilde{x}}
\bigr)
 -
\bigl(X_{s,t }^{\delta,x} 
-X_{\tilde{s},\tilde{t}}^{\delta,\tilde{x}} \bigr)\right\rVert_{\frac{p}{2}}\\
&\leq 
31 (b+c)(c+1)
\left[\sqrt{T}+p\right]^6
e^{5 c^2\left[\sqrt{T}+p\right]^2 T}\\
&\qquad\qquad
 e^{4.5\barC T/p}\frac{(V(x))^{\nicefrac{2}{p}}
+(V(\tilde{x}))^{\nicefrac{2}{p} }}{2}
\left[
\lvert s-\tilde{s}\rvert^{\nicefrac{1}{2}}+
\lvert t-\tilde{t}\rvert^{\nicefrac{1}{2}}+\lVert x-\tilde{x}\rVert\right]
\size{\delta}^{\nicefrac{1}{2}}.\end{split}\label{m32b}
\end{align}
\end{enumerate}
\end{theorem}

\begin{proof}[Proof of \cref{s01b}]
Throughout this proof let $\sigma_1,\sigma_2,\ldots,\sigma_m\in C(\R^d,\R^d)$ satisfy that $\sigma= (\sigma_1,\sigma_2,\ldots,\sigma_m)$.
First,
observe that
\eqref{v03} (applied with $(x,y,\tilde{x},\tilde{y})\gets (x,y,x,x)  $ in the notation of \eqref{v03}), the fact that
$\forall\,x\in\R^d\colon
\lVert\mu(0)\rVert+\lVert\sigma(0)\rVert+c\lVert x\rVert\leq  (V(x))^{\nicefrac{1}{p}}
$, and the triangle inequality
 prove for all $x,y\in\R^d$, $\zeta\in\{\mu,\sigma\}$ that
\begin{align}
\left\lVert
\zeta(x)-\zeta(y)\right\rVert
\leq c \lVert x-y\rVert
\quad\text{and}\quad
\lVert \zeta(x)\rVert\leq 
\lVert \zeta(0)\rVert+c\lVert x\rVert
\leq 
 (V(x))^{\nicefrac{1}{p}}.\label{r11}
\end{align}
This,  \eqref{m42}, and the fact that
$\forall\,A,B\in [0,\infty),\lambda\in (0,1)\colon A^\lambda B^{1-\lambda}\leq \lambda A+(1-\lambda)B$ imply for all $x,y\in\R^d$  that 
\begin{align}
&\left\lvert ((\totalD V)(y))(\mu(x))\right\rvert+\frac{1}{2}\left\lvert \sum_{k=1}^{m}
((\totalD^2 V)(y))(\sigma_k(x),\sigma_k(x))\right\rvert \nonumber \\
&\leq\xeqref{m42}
\barC(V(y))^{1-\frac{1}{p}}
\lVert \mu(x)\rVert+\frac{1}{2} 
\barC(V(y))^{1-\frac{2}{p}}\lVert \sigma(x)\rVert^2\nonumber\\
& 
\leq \xeqref{r11}
\barC (V(y))^{1-\frac{1}{p}}(V(x))^{\frac{1}{p}}+\frac{1}{2}\barC(V(y))^{1-\frac{2}{p}}
(V(x))^\frac{2}{p}  \nonumber
\\
&
\leq
\barC\left[\left(1-\frac{1}{p}\right)V(y)+\frac{1}{p}V(x)\right]+ \frac{1}{2}\barC\left[\left(1-\frac{2}{p}\right)V(y)+\frac{2}{p}V(x)\right] \nonumber \\
&= \left[\barC\left(1-\frac{1}{p}\right)+\frac{1}{2}\barC\left(1-\frac{2}{p}\right)\right]V(y)
+
\left[\barC\frac{1}{p}+\frac{1}{2}\barC
\frac{2}{p}\right]V(x)  \nonumber\\
&
= \left(1.5\barC-\frac{2 \barC}{p}\right)V(y)+\frac{2\barC}{p}V(x).
\label{r04}\end{align}
This  and
\cite[Lemma~2.2]{CoxHutzenthalerJentzen2014} (applied
for all $s\in[0,T)$,
$t\in(s,T]$,  $x\in\R^d$
with
$T\gets  t-s$,
$O\gets \R^d$,
$V\gets \bigl( [0,t-s]\times\R^d\ni(\mathfrak{t},\mathfrak{x})\mapsto V(\mathfrak{x})\in[0,\infty)\bigr)$, $\alpha\gets  \bigl([0,t-s]\ni \mathfrak{t}\mapsto 1.5 \barC \in [0,\infty) \bigr)$, 
$\tau\gets  t-s$,
$X\gets  (X^{\funcPi,x}_{s,s+r})_{r\in[0,t-s]}$  
 in the notation of
\cite[Lemma~2.2]{CoxHutzenthalerJentzen2014})  show for all
$x\in\R^d$, $s\in[0,T]$, $t\in [s,T]$ 
  that 
\begin{align}\label{r05}
\E \bigl[V(X_{s,t}^{\funcPi,x})\bigr]\leq e^{1.5\barC(t-s)}V(x).
\end{align}
Next, \eqref{r04},
 \cite[Theorem~2.4]{hudde2021stochastic} (applied for all $x\in\R^d$, $s\in[0,T]$ with
$H\gets  \R^d$, $U\gets \R^m$, $O\gets \R^d$,
$\tau\gets  (\Omega\ni\omega\mapsto s\in[0,T] )$,
$X\gets  (x+\mu(x)t+\sigma(x)W_{t})_{{t}\in [0,T]}$,
 $a\gets \bigl([0,T]\times\Omega\ni ({t},\omega)\mapsto \mu(x)\in\R^d \bigr)$,
 $b\gets \bigl([0,T]\times\Omega\ni (t,\omega)\mapsto \sigma(x)\in\R^{{d\times m}} \bigr)$,
$p\gets  1$,
$V\gets  ( [0,T]\times\R^d\ni (t,y))\mapsto V(y)\in[0,\infty))$,
 $\alpha\gets \bigl([0,T]\times\Omega\ni ({t},\omega)\mapsto  (1.5\barC-\frac{2\barC}{p})\in[0,\infty]\bigr)$,
$\beta\gets \bigl([0,T]\times\Omega\ni ({t},\omega)\mapsto  \frac{2\barC V(x)}{p}\in[0,\infty]\bigr)$, $q_1\gets  1$, $q_2\gets  \infty$  in the notation of \cite[Theorem~2.4]{hudde2021stochastic}), and
the fact that $\forall\,a\in[0,\infty)\colon 1+ a\leq e^{a}$ imply for all $x\in\R^d$, $s\in[0,T]$ that
\begin{align}\begin{split}
&
\E\bigl[
V(x+\mu(x)s+\sigma(x)W_{s}))\bigr]
\leq e^{(1.5\barC-\frac{2\barC}{p})s}\left(1+\tfrac{2\barC s}{p}\right)V(x)\\
&
\leq e^{(1.5\barC-\frac{2\barC}{p})s}e^{\frac{2\barC s}{p}}V(x)
= e^{1.5\barC s}V(x).\end{split}
\end{align}%
This, the tower property,
the disintegration theorem (see, e.g., \cite[Lemma 2.2]{HJK+18}),
 the Markov property of $W$, and the fact that
$\forall\, s\in[0,T],t\in[s,T], B\in\mathcal{B}(\R^d)\colon \P((W_{t}-W_{s})\in B)=\P( W_{t-s}\in B)$ imply  for all
$ \delta\in \setS$, $x\in\R^d$, $s\in[0,T]$, $t\in [s,T]$ 
  that
\begin{align}
&\E \bigl[V(X_{s,t}^{\delta,x} )\bigr]=\E\! \left[\E \bigl[V(X_{s,t}^{\delta,x})\big|\F_{\max\{s,\rdown{t}{\delta}\}}\bigr] \right]\nonumber 
\\&=\E\biggl[\E \Bigl[V\Bigl(z+
\mu(z)(t- \max\{s,\rdown{t}{\delta}\})+
\sigma(z)(W_{t} -W_{ \max\{s,\rdown{t}{\delta}\}} ) \Bigr)\Bigr|_{z=X_{s,\max\{s,\rdown{t}{\delta}\}}^{\delta,x} } \Big|\F_{\max\{s,\rdown{t}{\delta}\}}\Bigr]\biggr]\nonumber \\
&
=\E\!\left[\E \Bigl[V \Bigl(z+
\mu(z)(t- \max\{s,\rdown{t}{\delta}\})+
\sigma(z)(W_{t- \max\{s,\rdown{t}{\delta}\}} ) \Bigr) \Bigr]\Bigr|_{ z=X_{s,\max\{s,\rdown{t}{\delta}\}}^{\delta,x} }\right]\nonumber \\
&\leq e^{1.5\barC (t-\max\{s,\rdown{t}{\delta}\})}
\E\!\left[
V\bigl(X_{s,\max\{s,\rdown{t}{\delta}\}}^{\delta,x} \bigr)\right].
\end{align}
This,  induction, and \eqref{y01}  show 
for all $\delta\in\setS$,
 $x\in\R^d$, $s\in[0,T]$, $t\in [s,T]$ 
  that 
$
\E \bigl[V(X_{s,t}^{\delta,x})\bigr]\leq e^{1.5\barC (t-s)}V(x).
$ This, \eqref{r05}, and Jensen's inequality
 show 
for all 
$q\in[1,p]$,
$\delta\in\setSB$,
 $x\in\R^d$, $s\in[0,T]$, $t\in [s,T]$ 
  that 
\begin{equation}\label{r06}
\left \lVert (V(X_{s,t}^{\delta,x}))^{\nicefrac{1}{p}}\right\rVert_q\leq \bigl(e^{1.5\barC\lvert t-s\rvert}V(x)\bigr)^{\nicefrac{1}{p}}.
\end{equation}
This shows 
\eqref{r06b}.   

Next, 
H\"older's inequality,
the
Burkholder-Davis-Gundy inequality (see, e.g., \cite[Lemma~7.7]{dz92}), and
the fact that
$\forall\,t,\tilde{t}\in[0,T]\colon \lvert t-\tilde{t}\rvert^{\nicefrac{1}{2}}\leq\sqrt{T}$
show  for all 
$q\in [2,p]$,
$s\in[0,T]$,
$t\in[s,T]$, 
$\tilde{t}\in[t,T]$,
$ \mathfrak{a}\in \mathrm{span}_{\R}\{ ( \mu( X^{\delta,x}_{s,\max\{s,\delta(r) \}} ) )_{r\in [t,\tilde{t}]}\colon \delta\in \setSB,x\in \R^d\}$,
$ \mathfrak{b}\in \mathrm{span}_{\R}\{ ( \sigma( X^{\delta,x}_{s,\max\{s,\delta(r) \}} ) )_{r\in [t,\tilde{t}]}\colon \delta\in \setSB,x\in \R^d\}$
  that 
\begin{align}\label{m18}
\begin{split}
&
\left\lVert 
\int_{t}^{\tilde{t}}\mathfrak{a}_r\,dr\right\rVert_{q}+
\left\lVert 
\int_{t}^{\tilde{t}}\mathfrak{b}_r\,dW_r\right\rVert_{q}
\leq \lvert \tilde{t}-t\rvert^{\nicefrac{1}{2}}
\left[
\int_{t}^{\tilde{t}}\lVert \mathfrak{a}_r\rVert_{q}^2\,dr \right]^{\nicefrac{1}{2}}
+
q
\left[
\int_{t}^{\tilde{t}}\lVert \mathfrak{b}_r\rVert_{q}^2\,dr\right]^{\nicefrac{1}{2}} \\
&\leq \left[\sqrt{T}+{q}\right]\left[\int_{t}^{\tilde{t}}
\max\{\lVert \mathfrak{a}_r \rVert_q^2,\lVert \mathfrak{b}_r \rVert_q^2\}\,dr\right]^{\nicefrac{1}{2}}
\leq \left[
\sqrt{T}+{q}
\right]\lvert \tilde{t}-t\rvert^{\nicefrac{1}{2}}
\sup_{r\in\left[t,\tilde{t}\right]}\!\max\{\lVert \mathfrak{a}_r\rVert_{q},
\lVert \mathfrak{b}_r\rVert_{q} \}.\end{split} 
\end{align}
This,
\eqref{y01},
the triangle inequality,
 \eqref{r11}, and \eqref{r06} show for all 
$q\in [2,p]$,
$s\in[0,T]$,
$t\in[s,T]$, 
$\tilde{t}\in[t,T]$,
$ \delta\in \setSB$, $x\in \R^d$ 
that
\begin{equation}\begin{split}
& \left\lVert X _{s,\tilde{t}}^{\delta,x}-X _{s,t}^{\delta,x}\right\rVert_q
\leq 
\xeqref{y01}
\left\lVert \int_{t}^{\tilde{t}}\mu ( X^{\delta,x}_{s,\max\{s,\delta(r) \}} )\,dr\right\rVert_q
+\left\lVert \int_{t}^{\tilde{t}}\sigma ( X^{\delta,x}_{s,\max\{s,\delta(r) \}} )\,dW_{r}
\right\rVert_q \\
&
\leq \xeqref{m18}\left[
\sqrt{T}+q
\right]\lvert
\tilde{t}- t\rvert^{\nicefrac{1}{2}}\sup_{r\in\left[t,\tilde{t}\right]}\max_{\zeta\in\{ \mu,\sigma\}}\left\lVert \zeta ( X^{\delta,x}_{s,\max\{s,\delta(r) \}})\right\rVert_{q}
\\
&
\leq \xeqref{r11}\left[
\sqrt{T}+q
\right]\lvert \tilde{t}- t\rvert^{\nicefrac{1}{2}}
\sup_{r\in \left[t,\tilde{t}\right ]}\left\lVert V ( X^{\delta,x}_{s,\max\{s,\delta(r) \}} )^{\nicefrac{1}{p}}\right\rVert_q\\
&\leq\xeqref{r06} \left[
\sqrt{T}+q
\right]
(e^{1.5 \barC T}V(x))^{\nicefrac{1}{p} }\lvert \tilde{t}-t\rvert^{\nicefrac{1}{2}}
.\end{split}
\label{c03d}
\end{equation}
Moreover,  \eqref{m18} and
\eqref{r11} imply for all 
$q\in[2,p]$,
$s\in [0,T] $, $t\in [s,T]$, 
$\mathbb{X},\mathbb{Y}\in 
\{ (  X^{\delta,x}_{s,\max\{s,\delta(r) \}}  )_{r\in [s,t]}\colon \delta\in \setSB,x\in \R^d\}$
that
\begin{align}\begin{split}
&
\left\lVert \int_{s}^{t}
\mu(\mathbb{X}_{r})-
\mu(\mathbb{Y}_{r})\,dr\right\rVert_q
+\left\lVert \int_{s}^{t}
\sigma(\mathbb{X}_{r})-
\sigma(\mathbb{Y}_{r})
\,dW_{r}\right\rVert_{q}\\
&
\leq \xeqref{m18}
\left[
\sqrt{T}+q
\right]\left[
 \int_{s}^{t}\max_{\zeta\in \{\mu,\sigma\}}\lVert 
\zeta(\mathbb{X}_{r})-
\zeta(\mathbb{Y}_{r})\rVert_q^2 dr \right]^{\nicefrac{1}{2}}
 \\
&\leq  \xeqref{r11}c\left[
\sqrt{T}+q
\right]\left[
\int_{s}^{t}
\left\lVert 
\mathbb{X}_{r}-\mathbb{Y}_{r}\right\rVert_{q}^2dr\right]^{\nicefrac{1}{2}}
\leq c\left[
\sqrt{T}+q
\right] \lvert t-s\rvert^{\nicefrac{1}{2}}\sup_{r\in[s,t]}
\left\lVert 
\mathbb{X}_{r}-\mathbb{Y}_{r}\right\rVert_{q}
.\label{m01}\end{split}
\end{align}
This, \eqref{y01}, and
the triangle inequality
 show  for all 
$q\in[2,p]$,
$\delta\in \setSB$,
$s\in[0,T]$, $t\in[s,T]$,  $x,\tilde{x}\in\R^d$ 
 that\begin{equation}\begin{split}
&\left\lVert  X^{\delta,x}_{s,t}
- X^{\delta,\tilde{x}}_{s,t}\right\rVert_{q}
\leq \xeqref{y01}
\lVert 
x-\tilde{x}\rVert+\left\lVert \int_{s}^{t}\mu( X_{s,\max\{s,\rdown{r}{\delta}\}}^{\delta,x})
-\mu( X_{s,\max\{s,\rdown{r}{\delta}\}}^{\delta,\tilde{x}})
\,d{r}\right\rVert_{q}\\
&\qquad
+\left\lVert \int_{s}^{t}\sigma( X_{s,\max\{s,\rdown{r}{\delta}\}}^{\delta,x})
-\sigma( X_{s,\max\{s,\rdown{r}{\delta}\}}^{\delta,\tilde{x}})
\,dW_{r}\right\rVert_{q}\\
&\leq \lVert x-\tilde{x}\rVert
+\xeqref{m01} c\left[
\sqrt{T}+q
\right]\left[
\int_{s}^{t}
\left\lVert 
X^{\delta,x}_{s,\max\{s,\rdown{r}{\delta}\}}-
X^{\delta,\tilde{x}}_{s,\max\{s,\rdown{r}{\delta}\}}\right\rVert_{q}^2dr\right]^{\nicefrac{1}{2}}.
\end{split}\label{k02}\end{equation}
This, \cref{b02}, and integrability in \eqref{r06b}
imply for all 
$\delta\in \setSB$, $q\in[2,p]$,
$s\in[0,T]$, $t\in[s,T]$,  $x,\tilde{x}\in\R^d$ 
 that 
\begin{equation}\label{m24}
\left\lVert
X_{s,t}^{\delta,x}-
X_{s,t}^{\delta,\tilde{x}}
\right\rVert_{q}
\leq 
\sqrt{2}\lVert x-\tilde{x} \rVert
e^{c^2\left[
\sqrt{T}+q
\right]^2 T}.
\end{equation}
Next, \eqref{y01} 
proves  for all $\delta\in\setSB$,
$s\in[0,T]$, $t\in[s,T]$, 
$x\in\R^d$ that 
\begin{align}\begin{split}
&    X_{s,t}^{\delta,x}
-X_{s,t}^{\funcPi,x}  
=    \xeqref{y01}
\int_{s}^{t}
\mu(X_{s,\max\{s,\rdown{r}{\delta}\}}^{\delta,x})
-
\mu(X_{s,r}^{\funcPi,x})
\,dr+
\int_{s}^{t}
\sigma(X_{s,\max\{s,\rdown{r}{\delta}\}}^{\delta,x})
-\sigma(X_{s,r}^{\funcPi,x})
\,dW_{r}  
\\
&=     \int_{s}^{t}
\mu(X_{s,\max\{s,\rdown{r}{\delta}\}}^{\delta,x})
-\mu(X_{s,\max\{s,\rdown{r}{\delta}\}}^{\funcPi,x})
\,dr   +
  \int_{s}^{t}
\sigma(X_{s,\max\{s,\rdown{r}{\delta}\}}^{\delta,x})
-\sigma(X_{s,\max\{s,\rdown{r}{\delta}\}}^{\funcPi,x})
\,dW_{r}   \\
&\quad +   
\int_{s}^{t}
\mu(X_{s,\max\{s,\rdown{r}{\delta}\}}^{\funcPi,x})
-\mu(X_{s,r}^{\funcPi,x})
\,dr   +  
\int_{s}^{t}
\sigma(X_{s,\max\{s,\rdown{r}{\delta}\}}^{\funcPi,x})
-\sigma(X_{s,r}^{\funcPi,x})
\,dW_{r}   .\end{split}\label{m27}
\end{align}
This,
the triangle inequality, 
 \eqref{m01}, 
 \eqref{c03d}, and
 the fact that $\forall\, t\in[0,T],\delta\in\setS\colon
0\leq t- \delta(t)\leq \size{\delta}$
prove for all $\delta\in\setSB$,
$q\in[2,p]$,
$s\in[0,T]$, $t\in[s,T]$, 
$x\in\R^d$ that 
\begin{align}\begin{split}
&\left \lVert X_{s,t}^{\delta,x}
-X_{s,t}^{\funcPi,x}\right\rVert_{q}
\leq \xeqref{m27}\xeqref{m01} c\left[
\sqrt{T}+q
\right]\Biggl[\left( \int_{s}^{t}
\bigl \lVert 
X_{s,\max\{s,\rdown{r}{\delta}\}}^{\delta,x}
-X_{s,\max\{s,\rdown{r}{\delta}\}}^{\funcPi,x}\bigr\rVert_{q}^2\,dr   \right)^{\nicefrac{1}{2}}\\
&\qquad\qquad\qquad\qquad\qquad\qquad\qquad\qquad\qquad
+
\lvert t-s\rvert^{\nicefrac{1}{2}}
\sup_{r\in[s,t]}
\bigl \lVert 
X_{s,\max\{s,\rdown{r}{\delta}\}}^{\funcPi,x}
-X_{s,r}^{\funcPi,x}
\bigr\rVert_{q}\Biggr] \\
&\leq  c\left[
\sqrt{T}+q
\right]
\left(\int_{s}^{t}
\bigl \lVert 
X_{s,\max\{s,\rdown{r}{\delta}\}}^{\delta,x}
-X_{s,\max\{s,\rdown{r}{\delta}\}}^{\funcPi,x}\bigr\rVert_{q}^2\,dr\right)^{\nicefrac{1}{2}}\\
&\qquad\qquad\qquad\qquad\qquad
+
c\left[\sqrt{T}+q\right]\lvert {t}-s\rvert^{\nicefrac{1}{2}}\xeqref{c03d} \left[\sqrt{T}+q\right]
(e^{1.5\barC T}V(x))^{\nicefrac{1}{p} } \size{\delta}^{\nicefrac{1}{2}}.\label{m12b}
\end{split}\end{align}
This, \cref{b02},  integrability in \eqref{r06b}, and the fact that
$\forall\,t,s\in[0,T]\colon \lvert t-s\rvert^{\nicefrac{1}{2}}\leq\sqrt{T}+p $
show for all
$q\in[2,p]$,
 $\delta\in\setSB$,
$s\in[0,T]$, 
$t\in [s,T]$,
$x\in\R^d$ that 
\begin{equation}\begin{split}
&\left\lVert 
X_{s,t}^{\delta,x}
-X_{s,t}^{\funcPi,x}\right\rVert_{q}
\leq 
 \sqrt{2}e^{ c^2\left[\sqrt{T}+q\right]^2 T}\xeqref{m12b}
c\left[
\sqrt{T}+q
\right]^2
(e^{1.5 \barC T}V(x))^{\nicefrac{1}{p} }
\lvert t-s\rvert^{\nicefrac{1}{2}} \size{\delta}^{\nicefrac{1}{2}}\\
&\leq \sqrt{2}e^{ c^2\left[\sqrt{T}+q\right]^2 T}
c\left[
\sqrt{T}+q
\right]^3
(e^{1.5 \barC T}V(x))^{\nicefrac{1}{p} }
\size{\delta}^{\nicefrac{1}{2}}
.\label{m12}
\end{split}\end{equation}
This
proves \eqref{m29}.

\begin{figure}\center
\begin{subfigure}[b]{0.3\textwidth}\centering
\begin{tikzpicture}
\draw (0,0) -- (5,0);
\node [above]  at (.5,0) {$s$};
\node [above]  at (1.5,0) {$\tilde{s}$};
\node [above] at (4,0) {$\bar{s}$};
\draw plot [mark=+]  (.5,0);
\draw plot [mark=+]  (1.5,0);
\draw plot [mark=+]  (4,0);
\end{tikzpicture}\subcaption{There is no grid point on $(s,\bar{s})$.}
\label{f01a}
\end{subfigure}\hspace{2cm}
\begin{subfigure}[b]{0.3\textwidth}
\begin{tikzpicture}
\draw (0,0) -- (5,0);
\node [above]  at (.5,0) {$s$};
\node [above]  at (1.5,0) {$\tilde{s}$};
\node [above]  at (2.5,0) {$\bar{s}$};
\node [above] at (4,0) {$t$};
\draw plot [mark=+]  (.5,0);
\draw plot [mark=+]  (1.5,0);
\draw plot [mark=x]  (2.5,0);
\draw plot [mark=+]  (4,0);
\end{tikzpicture}\subcaption{$\bar{s}$ is the smallest grid point on $(\tilde{s},t]$.}\label{f01b}
\end{subfigure}

\begin{subfigure}[b]{0.3\textwidth}
\begin{tikzpicture}
\draw (0,0) -- (5,0);
\node [above]  at (.5,0) {$s$};
\node [above]  at (2.5,0) {$\tilde{s}$};
\node [above] at (4,0) {$t$};
\draw plot [mark=+]  (.5,0);
\draw plot [mark=x]  (2.5,0);
\draw plot [mark=+]  (4,0);
\end{tikzpicture}
\subcaption{$\tilde{s}$ is a grid point.}
\label{f01c}
\end{subfigure}\hspace{2cm}
\begin{subfigure}[b]{0.3\textwidth}
\begin{tikzpicture}
\draw (0,0) -- (5,0);
\node [above]  at (.5,0) {$s$};
\node [above]  at (2.5,0) {$\bar{s}$};
\node [above]  at (3.5,0) {$\tilde{s}$};
\node [above] at (4.5,0) {$t$};
\draw plot [mark=+]  (.5,0);
\draw plot [mark=x]  (2.5,0);
\draw plot [mark=+]  (3.5,0);
\draw plot [mark=+]  (4.5,0);
\end{tikzpicture}\subcaption{$\bar{s}$ is the largest grid point on $[s,\tilde{s})$.}
\label{f01d}\end{subfigure}
\caption{An illustration for the case distinction.
A  grid point is drawn by
 $\times$. 
}%
\end{figure}
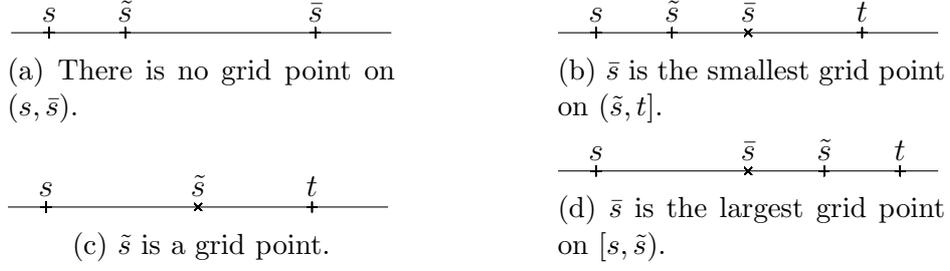

Next, the Markov property of 
the exact solution in \eqref{y01},
 the fact  that the Euler approximations \eqref{y01}
restricted to their grid points satisfy the Markov property,
the disintegration theorem (see, e.g., \cite[Lemma 2.2]{HJK+18}),
\eqref{m24}, and \eqref{c03d} show
for all 
$q\in[2,p]$,
$\delta\in \setSB$, $x\in\R^d$,
$s\in[0,T] $, 
$\tilde{s}\in\delta([0,T])\cap [s,T]$,
$t\in[\tilde{s},T]$ (cf. \cref{f01c}) that
\begin{align}\begin{split}
&\left\lVert X_{s,t}^{\delta,x}-X_{\tilde{s},t}^{\delta,x}\right\rVert _q
=
 \left\lVert \left\lVert X_{\tilde{s},t}^{\delta,\eta}-X_{\tilde{s},t}^{\delta,x}\right\rVert _q\bigr|_{\eta=X_{s,\tilde{s}}^{\delta,x} }\right\rVert _q 
\\
&\leq \xeqref{m24}
\sqrt{2}e^{c^2\left[
\sqrt{T}+q
\right]^2 T}\left\lVert X_{s,\tilde{s}}^{\delta,x} -x \right\rVert_{q}
\leq \sqrt{2}e^{c^2\left[
\sqrt{T}+q
\right]^2 T} \xeqref{c03d}
\left[
\sqrt{T}+q
\right]
(e^{1.5 \barC T }V(x))^{\nicefrac{1}{p} }\lvert \tilde{s}-s\rvert^{\nicefrac{1}{2}}
.\label{s06b}\end{split}
\end{align}
Next, \eqref{y01}, the triangle inequality, 
and \eqref{c03d} 
show that
for all $q\in[2,p]$,
$\delta\in \setS$,
$s\in [0,T]$, 
$\tilde{s}\in[s,T]$,
$\bar{s}\in[\tilde{s},T]$ 
with 
$(s,\bar{s})\cap \delta([0,T]) =\emptyset$ (cf. \cref{f01a})
it holds that
\begin{equation}
X_{s,\bar{s}}^{\delta,x}= \mu(x)(\bar{s}-s) + 
\sigma(x)(W_{\bar{s}}-W_{s}),\quad 
X_{\tilde{s},\bar{s}}^{\delta,x}= \mu(x)(\bar{s}-\tilde{s}) + 
\sigma(x)(W_{\bar{s}}-W_{\tilde{s}}),
\end{equation}
 and 
\begin{align}
\left \lVert  X_{\tilde{s},\bar{s}}^{\delta,x}- X_{s,\bar{s}}^{\delta,x}\right\rVert_{q}=\left \lVert 
\mu(x)(s-\tilde{s}) + \sigma(x)(W_{s}-W_{\tilde{s}})\right\rVert_{q}
\leq \xeqref{c03d}\left[
\sqrt{T}+q
\right](e^{1.5 \barC T } V(x))^{\nicefrac{1}{p}}\lvert s-\tilde{s}\rvert^{\nicefrac{1}{2}}.
\label{m40}
\end{align}
Furthermore, note that (cf. \cref{f01a,f01b}) for all $\delta\in \setS$,
$s\in [0,T]$, 
$\tilde{s}\in[s,T]$,
$t\in [\tilde{s},T]$
with $ (s,\tilde{s})\cap \delta([0,T])=\emptyset$
 there exists $\bar{s}\in[\tilde{s},t]\cap\bigl(\delta([0,T])\cup \{t\}\bigr)$
with 
$(s,\bar{s})\cap \delta([0,T]) =\emptyset$.
This,   the fact  that the Euler approximations \eqref{y01}
restricted to their grid points satisfy the Markov property,
the disintegration theorem (see, e.g., \cite[Lemma 2.2]{HJK+18}),
 \eqref{m24}, and \eqref{m40} prove that
for all 
$q\in[2,p]$,
$\delta\in \setS$,
$s\in [0,T]$, 
$\tilde{s}\in[s,T]$,
$t\in [\tilde{s},T]$,
$\bar{s}=\min( [\tilde{s},t]\cap\bigl(\delta([0,T])\cup \{t\}\bigr) )$
(cf. \cref{f01a,f01b}) it holds that
\begin{align}
&\left\lVert X_{s,t}^{\delta,x}-X_{\tilde{s},t}^{\delta,x}\right\rVert_q=
\left\lVert \left\lVert X_{\bar{s},t}^{\delta,\eta}-X_{\bar{s},t}^{\delta,\tilde{\eta}}\right\rVert_q\bigr|_{
\eta=X_{s,\bar{s}}^{\delta,x},\, \tilde{\eta}=X_{\tilde{s},\bar{s}}^{\delta,x}}\right\rVert_q \nonumber\\
&\leq \xeqref{m24}
\left\lVert 
\left[\sqrt{2}e^{c^2\left[
\sqrt{T}+q
\right]^2 T}\lVert \eta-\tilde{\eta} \rVert\right]
\bigr|_{\eta=X_{s,\bar{s}}^{\delta,x},\, \tilde{\eta}=X_{\tilde{s},\bar{s}}^{\delta,x}}
\right\rVert_{q}\nonumber\\
& \leq \sqrt{2}e^{c^2\left[
\sqrt{T}+q
\right]^2 T}\xeqref{m40}
\left[
\sqrt{T}+q
\right]
(e^{1.5 \barC T } V(x))^{\nicefrac{1}{p}}|\tilde{s}-s|^{\nicefrac{1}{2}}.\label{m09}
\end{align}
This, the triangle inequality, and \eqref{s06b} prove that for all 
$q\in[2,p]$,
$\delta\in \setS$,
$s\in [0,T]$, 
$\tilde{s}\in[s,T]$,
$t\in [\tilde{s},T]$, $\bar{s}\in [s,\tilde{s}]\cap (\delta([0,T]) \cup\{s\})$ with 
$ (\bar{s},\tilde{s})\cap\delta([0,T])=\emptyset$ (cf. \cref{f01d}) it holds that
$\max\{ \lvert s-\bar{s}\rvert, \lvert \bar{s}-\tilde{s}\rvert\}\leq \lvert s-\tilde{s}\rvert $ and
\begin{align}\begin{split}
&\left\lVert   X_{s,t}^{\delta,x}-X_{\tilde{s},t}^{\delta,x} 
\right\rVert_{q}
\leq 
\left\lVert  X_{s,t}^{\delta,x}-X_{\bar{s},t}^{\delta,x} \right\rVert_{q}
+
\left\lVert 
  X_{\bar{s},t}^{\delta,x}-X_{\tilde{s},t}^{\delta,x} 
\right\rVert_{q}\\
&\leq 2\sqrt{2}e^{c^2\left[
\sqrt{T}+q
\right]^2T}
\left[
\sqrt{T}+q
\right]
(e^{1.5 \barC T}V(x))^{\nicefrac{1}{p} }\lvert s-\tilde{s}\rvert^{\nicefrac{1}{2}}.
\end{split}\label{m11b}\end{align}
Next, note that (cf. \cref{f01d}) for all $\delta\in \setS$,
$s\in [0,T]$, 
$\tilde{s}\in[s,T]$,
$t\in [\tilde{s},T]$ there exists $\bar{s}\in [s,\tilde{s}]\cap (\delta([0,T]) \cup\{s\})$ with 
$ (\bar{s},\tilde{s})\cap\delta([0,T])=\emptyset$. 
This, \eqref{m11b}, symmetry, and the strong convergence in \eqref{m29} as $\size{\delta}\to 0$ 
 imply
that for all 
$q\in[2,p]$,
$\delta\in \setSB$,
$s,\tilde{s}\in[0,T]$,
$t \in [\max\{s,\tilde{s}\},T] $ it holds that
\begin{align}
&\left\lVert   X_{s,t}^{\delta,x}-X_{\tilde{s},t}^{\delta,x} 
\right\rVert_{q}
\leq {2}\sqrt{2}e^{c^2\left[
\sqrt{T}+q
\right]^2T}
\left[
\sqrt{T}+q
\right]
(e^{1.5 \barC T}V(x))^{\nicefrac{1}{p} }\lvert s-\tilde{s}\rvert^{\nicefrac{1}{2}}.\label{m11}
\end{align}
Next, the fact that $\forall\, s,\tilde{s},t,\tilde{t}\in [0,T]\colon
\max\{\tilde{s},\tilde{t}\}\leq  \max\{s,\tilde{s},t,\tilde{t}\}\leq  \max\{\tilde{s},\tilde{t}\}
+
\lvert s-\tilde{s}\rvert+\lvert  t-\tilde{t}\rvert
$ and
the fact that
$\forall A,B\in [0,\infty)\colon \sqrt{A+B}\leq 
\sqrt{A}+\sqrt{B}$ show
 for all $s,\tilde{s},t,\tilde{t}\in[0,T]$ 
that
\begin{equation} \label{m31}
\left\lvert \max\{s,\tilde{s},t,\tilde{t}\}-\max\{\tilde{s},\tilde{t}\}\right\rvert^{\nicefrac{1}{2}}
\leq \lvert s-\tilde{s}\rvert^{\nicefrac{1}{2}}+\lvert  t-\tilde{t}\rvert^{\nicefrac{1}{2}}.
\end{equation}
This, the triangle inequality, \eqref{c03d},  \eqref{m11}, \eqref{m24}, 
the fact that $V\geq 1$,
and the fact that
$2\sqrt{2}+2\leq 5$
show that for all $\delta\in \setSB$,
$ s,\tilde{s},t,\tilde{t}\in [0,T]$, $q\in [2,p]$, $x,\tilde{x}\in\R^d$
with $ V(x)\leq V(\tilde{x})$
it holds
 that
\begin{align}
&\left \lVert X^{\delta,x}_{s,\max\{s,t\}}
-
X^{\delta,\tilde{x}}_{\tilde{s},\max\{\tilde{s},\tilde{t}\}}\right\rVert_{q}
\leq  
\left \lVert 
X^{\delta,x}_{s,\max\{s,t\}}-
X^{\delta,x}_{s,\max\{s,\tilde{s},t,\tilde{t}\}}\right\rVert_{q}
+\left \lVert 
X^{\delta,x}_{s,\max\{s,\tilde{s},t,\tilde{t}\}}
-
X^{\delta,x}_{\tilde{s},\max\{s,\tilde{s},t,\tilde{t}\}}\right\rVert_{q} \nonumber\\
&\qquad\qquad
+\left \lVert 
X^{\delta,x}_{\tilde{s},\max\{s,\tilde{s},t,\tilde{t}\}}
-X^{\delta,x}_{\tilde{s},\max\{\tilde{s},\tilde{t}\}}
\right\rVert_{q}
+\left \lVert 
X^{\delta,x}_{\tilde{s},\max\{\tilde{s},\tilde{t}\}}
-X^{\delta,\tilde{x}}_{\tilde{s},\max\{\tilde{s},\tilde{t}\}}\right\rVert_{q}
\nonumber
 \\
&\leq \xeqref{c03d}
\left[
\sqrt{T}+q
\right]
(e^{1.5 \barC T}V(x))^{\nicefrac{1}{p} }\lvert \max\{s,t\}-\max\{s,\tilde{s},t,\tilde{t}\}\rvert^{\nicefrac{1}{2}}
\nonumber\\
&\quad 
+\xeqref{m11}
2\sqrt{2}e^{c^2\left[\sqrt{T}+q\right]^2T}
\left[\sqrt{T}+q\right]
(e^{1.5 \barC T}V(x))^{\nicefrac{1}{p} }\lvert \tilde{s}-s\rvert^{\nicefrac{1}{2}}
\nonumber \\
&\quad +
\xeqref{c03d}
\left[\sqrt{T}+q\right]
(e^{1.5 \barC T}V(x))^{\nicefrac{1}{p} }\lvert \max\{s,\tilde{s},t,\tilde{t}\}-\max\{\tilde{s},\tilde{t}\}\rvert^{\nicefrac{1}{2}} 
+\xeqref{m24}\sqrt{2}\lVert x-\tilde{x} \rVert 
e^{c^2\left[
\sqrt{T}+q
\right]^2 T}\nonumber
 \\
&\leq
{5} e^{c^2\left[
\sqrt{T}+q
\right]^2T}
\left[
\sqrt{T}+q
\right]
e^{1.5 \barC T/p}
\frac{(V(x))^{\nicefrac{1}{p}}
+(V(\tilde{x}))^{\nicefrac{1}{p} }}{2}
\left[
\lvert s-\tilde{s}\rvert^{\nicefrac{1}{2}} +
\lvert t-\tilde{t}\rvert^{\nicefrac{1}{2}}\right]\nonumber\\
&\quad +
\sqrt{2}\lVert x-\tilde{x} \rVert 
e^{c^2\left[
\sqrt{T}+q
\right]^2 T}
.\label{m26}
\end{align}
This and symmetry establish \eqref{m28}.

Next,
\eqref{y01}, the triangle inequality, \eqref{m01}, and \eqref{m24}
prove for all 
$q\in [2,p]$,
$\delta\in\setSB$, $s\in[0,T]$,
$t \in[s,T]$, 
$\tilde{t}\in [t,T]$,
$x,\tilde{x}\in\R^d$ that
\begin{align}
&\left \lVert 
\bigl(X^{\delta,x}_{s,\tilde{t}}-X_{s,t}^{\delta,x}\bigr)
-\bigl(X^{\delta,\tilde{x}}_{s,\tilde{t}}-X^{\delta,\tilde{x}}_{s,t}
\bigr)\right\rVert_{q}\nonumber  \\&=
\xeqref{y01}
\left \lVert 
\int_{t}^{\tilde{t}}
\mu( X_{s,\max\{s,\rdown{r}{\delta}\}}^{\delta,x})
-\mu( X_{s,\max\{s,\rdown{r}{\delta}\}}^{\delta,\tilde{x}})
\,d{r}
+\int_{t}^{\tilde{t}}
\sigma( X_{s,\max\{s,\rdown{r}{\delta}\}}^{\delta,x})-
\sigma( X_{s,\max\{s,\rdown{r}{\delta}\}}^{\delta,\tilde{x}})
\,dW_{r}
\right\rVert_{q} \nonumber\\
&
\leq\xeqref{m01} c\left[
\sqrt{T}+q
\right]\lvert t-\tilde{t}\rvert^{\nicefrac{1}{2}}
\sup_{r\in[0,T]}
\left \lVert 
 X_{s,\max\{s,\rdown{r}{\delta}\}}^{\delta,x}
- X_{s,\max\{s,\rdown{r}{\delta}\}}^{\delta,\tilde{x}}\right\rVert_{q}\nonumber\\
&
\leq\xeqref{m24} c\left[
\sqrt{T}+q
\right]\sqrt{2}
e^{c^2\left[
\sqrt{T}+q
\right]^2 T} \lVert x-\tilde{x}\rVert\lvert t-\tilde{t}\rvert^{\nicefrac{1}{2}}
.\label{m14}
\end{align}
This and symmetry
prove \eqref{m36}.
 
For the rest of this proof we assume that $p\geq 4$.
The triangle inequality, \eqref{c03d}, \eqref{m12}, and \eqref{m24}
show  for all
$\delta\in\setSB$, $s\in[0,T]$,
$t\in[s,T]$, 
$x,y\in\R^d$ that
\begin{align}
&
\left\lVert 
X_{s,t}^{\funcPi,x}-
X_{s,\max\{s,\delta(t)\}}^{\delta,y}
\right\rVert_{p}\nonumber\\
&\leq 
\left\lVert 
X_{s,t}^{\funcPi,x}-X_{s,\max\{s,\delta(t)\}}^{\funcPi,x}\right\rVert_p+
\left\lVert 
X_{s,\max\{s,\delta(t)\}}^{\funcPi,x}-
X_{s,\max\{s,\delta(t)\}}^{\delta,x}
\right\rVert_{p}+
\left\lVert 
X_{s,\max\{s,\delta(t)\}}^{\delta,x}-
X_{s,\max\{s,\delta(t)\}}^{\delta,y}
\right\rVert_{p} \nonumber\\
&
\leq \xeqref{c03d}\left[
\sqrt{T}+p
\right]
(e^{1.5 \barC T}V(x))^{\nicefrac{1}{p} }\size{\delta}^{\nicefrac{1}{2}}
 +\xeqref{m12}
\sqrt{2}
\left[\sqrt{T}+p\right]^3c
e^{ c^2\left[\sqrt{T}+p\right]^2 T}
(e^{1.5 \barC T}V(x))^{\nicefrac{1}{p} }
\size{\delta}^{\nicefrac{1}{2}}\nonumber\\
&\quad 
+\xeqref{m24}
 \sqrt{2}e^{c^2\left[
\sqrt{T}+p
\right]^2 T}\lVert x-y \rVert\nonumber\\
&\leq \left[\sqrt{2}c+1\right]
\left[
\sqrt{T}+p
\right]^3
e^{ c^2\left[\sqrt{T}+p\right]^2 T}
(e^{1.5 \barC T}V(x))^{\nicefrac{1}{p} } 
\size{\delta}^{\nicefrac{1}{2}}+\sqrt{2}e^{c^2\left[
\sqrt{T}+p
\right]^2 T}\lVert x-y \rVert.\label{m15}
\end{align}
This, \eqref{v03},  the triangle inequality,
H\"older's inequality, the fact that $p\geq 4$,
\eqref{m14}, the fact that
$\forall\, t\in[0,T],\delta\in\setSB\colon
0\leq t- \delta(t)\leq \size{\delta}$,
 \eqref{m24}, and the fact that $V\geq 1$
show that for all
$\zeta\in \{\mu,\sigma\}$,
$\delta\in\setSB$, $s\in[0,T]$,
$t,\tilde{t}\in[s,T]$,
$x,y\in\R^d$ with $\tilde{t}= \max\{s,\delta(t)\}$ it holds that
\begin{align}
&
\left\lVert 
\bigl(
\zeta(X_{s,t}^{\funcPi,x})
-
\zeta(X_{s,\tilde{t}}^{\delta,y})
\bigr)
-
\bigl(
\zeta(X_{s,t}^{\funcPi,\tilde{x}})-
\zeta(X_{s,\tilde{t}}^{\delta,\tilde{y}})
\bigr)
\right\rVert_{\frac{p}{2}}
  \nonumber \\
&\leq \xeqref{v03}
c
\left\lVert 
\bigl(X_{s,t}^{\funcPi,x}
-
X_{s,\tilde{t}}^{\delta,y}
\bigr)
-
\bigl(
X_{s,t}^{\funcPi,\tilde{x}}
-X_{s,\tilde{t}}^{\delta,\tilde{y}}\bigr)
\right\rVert_{\frac{p}{2}}
+
b   
 \frac{\left\lVert 
X_{s,t}^{\funcPi,x}-X_{s,\tilde{t}}^{\delta,y}
\right\rVert_{p}+
\left\lVert 
X_{s,t}^{\funcPi,\tilde{x}}-X_{s,\tilde{t}}^{\delta,\tilde{y}}
\right\rVert_{p}}{2} 
\left\lVert 
X_{s,t}^{\funcPi,x}-X_{s,t}^{\funcPi,\tilde{x}}
\right\rVert_{p}
\nonumber 
\\
&\leq 
c
\left\lVert 
\bigl(X_{s,\tilde{t}}^{\funcPi,x}-X_{s,\tilde{t}}^{\delta,y}\bigr)
-
\bigl(
X_{s,\tilde{t}}^{\funcPi,\tilde{x}}-X_{s,\tilde{t}}^{\delta,\tilde{y}}\bigr)
\right\rVert_{\frac{p}{2}}
+
c
\left\lVert 
\bigl(
X_{s,t}^{\funcPi,x}-X_{s,\tilde{t}}^{\funcPi,x}\bigr)
-
\bigl(X_{s,t}^{\funcPi,\tilde{x}}-X_{s,\tilde{t}}^{\funcPi,\tilde{x}}\bigr)
\right\rVert_{\frac{p}{2}}\nonumber
\\
&\quad 
+
b    \frac{\left\lVert 
X_{s,t}^{\funcPi,x}-X_{s,\tilde{t}}^{\delta,y}
\right\rVert_{p}+
\left\lVert 
X_{s,t}^{\funcPi,\tilde{x}}-X_{s,\tilde{t}}^{\delta,\tilde{y}}
\right\rVert_{p}}{2} 
\left\lVert 
X_{s,t}^{\funcPi,x}-X_{s,t}^{\funcPi,\tilde{x}}
\right\rVert_{p}
\nonumber 
\\
&\leq c
\left\lVert 
\bigl(X_{s,\tilde{t}}^{\funcPi,x}-X_{s,\tilde{t}}^{\delta,y}\bigr)
-
\bigl(
X_{s,\tilde{t}}^{\funcPi,\tilde{x}}-X_{s,\tilde{t}}^{\delta,\tilde{y}}\bigr)
\right\rVert_{\frac{p}{2}}
+\xeqref{m14}
 c^2\left[\sqrt{T}+p\right]\sqrt{2}
e^{c^2\left[
\sqrt{T}+p
\right]^2 T} \lVert x-\tilde{x}\rVert\size{\delta}^{\nicefrac{1}{2}}
\nonumber\\
&\quad 
+ b\xeqref{m15}  \Biggl[\frac{\left[\sqrt{2}c+1\right]\left[\sqrt{T}+p\right]^3 e^{c^2\left[
\sqrt{T}+p
\right]^2  T}
e^{1.5 \barC T /p}\left(
(V(x))^{\nicefrac{1}{p}}
+(V(\tilde{x}))^{\nicefrac{1}{p}}\right)
\size{\delta}^{\nicefrac{1}{2}} }{2} \nonumber \\
&\qquad\qquad+
 \sqrt{2}e^{c^2\left[
\sqrt{T}+p
\right]^2  T} 
\frac{\lVert x-y \rVert+
\lVert \tilde{x}-\tilde{y} \rVert}{2}
\Biggr]\xeqref{m24}
 \sqrt{2}e^{c^2\left[
\sqrt{T}+p
\right]^2  T}\lVert x-\tilde{x} \rVert
\nonumber \\
&\leq c
\left\lVert 
\bigl(X_{s,\tilde{t}}^{\funcPi,x}-X_{s,\tilde{t}}^{\delta,y}\bigr)
-
\bigl(
X_{s,\tilde{t}}^{\funcPi,\tilde{x}}-X_{s,\tilde{t}}^{\delta,\tilde{y}}\bigr)
\right\rVert_{\frac{p}{2}}\nonumber\\
&\quad+
2(c^2+bc+b)
\left[\sqrt{T}+p\right]^3
e^{2 c^2\left[\sqrt{T}+p\right]^2 T}
e^{1.5 \barC T /p}
\frac{(V(x))^{\nicefrac{1}{p}}
+(V(\tilde{x}))^{\nicefrac{1}{p} }}{2}
\size{\delta}^{\nicefrac{1}{2}}\lVert x-\tilde{x} \rVert\nonumber \\
&\quad + 2b e^{2c^2\left[
\sqrt{T}+p
\right]^2 T}
\frac{\left(\lVert x-y \rVert+
\lVert \tilde{x}-\tilde{y} \rVert\right) \lVert x-\tilde{x}\rVert}{2}
. \label{m19}
\end{align}
Next, \eqref{y01}
proves for all
$\delta\in\setSB$, $s\in[0,T]$,
$t\in[s,T]$, 
$x,\tilde{x},y,\tilde{y}\in\R^d$ that
\begin{align}\begin{split}
&
\bigl(X^{\funcPi,x}_{s,t}-X^{\delta,y}_{s,t}\bigr)
-\bigl(
X^{\funcPi,\tilde{x}}_{s,t}-X^{\delta,\tilde{y}}_{s,t}\bigr)
=
\bigl(x-y \bigr)-\bigl(\tilde{x}-\tilde{y} \bigr)
\\
&+
\int_{s}^{t}
\bigl(
\mu(X_{s,r}^{\funcPi,x})-
\mu(X_{s,\max\{s,\rdown{r}{\delta}\}}^{\delta,y})
\bigr)
-
\bigl(
\mu(X_{s,r}^{\funcPi,\tilde{x}})
-\mu(X_{s,\max\{s,\rdown{r}{\delta}\}}^{\delta,\tilde{y}})
\bigr)\,
dr 
\\
&\qquad
+
\int_{s}^{t}
\bigl(
\sigma(X_{s,r}^{\funcPi,x})
-
\sigma(X_{s,\max\{s,\rdown{r}{\delta}\}}^{\delta,y})
\bigr)
-
\bigl(
\sigma(X_{s,r}^{\funcPi,\tilde{x}})
-
\sigma(X_{s,\max\{s,\rdown{r}{\delta}\}}^{\delta,\tilde{y}})
\bigr)
\,dW_{r}.\end{split} 
\end{align}
This, the triangle inequality,
\eqref{m18}, 
and
\eqref{m19}
prove for all
$\delta\in\setSB$, $s\in[0,T]$,
$t\in[s,T]$, 
$x,\tilde{x},y,\tilde{y}\in\R^d$ that
\begin{align}
&\left\lVert 
\bigl(X^{\funcPi,x}_{s,t}-X^{\delta,y}_{s,t}\bigr)
-\bigl(
X^{\funcPi,\tilde{x}}_{s,t}-X^{\delta,\tilde{y}}_{s,t}\bigr)\right\rVert_{\frac{p}{2}}
\leq 
 \left\lVert 
\bigl(
x-y \bigr)
-
\bigl(
\tilde{x}-\tilde{y} \bigr)\right\rVert\nonumber
\\
&\quad \xeqref{m18}+\left[\sqrt{T}+p\right]\Biggl[\int_{s}^{t}\max_{\zeta\in \{\mu,\sigma\}}
\Bigl \lVert \bigl(
\zeta(X_{s,r}^{\funcPi,x})
-
\zeta(X_{s,\max\{s,\rdown{r}{\delta}\}}^{\delta,y})
\bigr)\nonumber\\[-12pt]
&\qquad\qquad\qquad\qquad\qquad\qquad\qquad\qquad\qquad
-
\bigl(
\zeta(X_{s,r}^{\funcPi,\tilde{x}})
-
\zeta(X_{s,\max\{s,\rdown{r}{\delta}\}}^{\delta,\tilde{y}})
\bigr)\Bigr \rVert^2_{\frac{p}{2}}\,
dr\Biggr]^{\nicefrac{1}{2}}\nonumber
 \\
&\leq\xeqref{m19} \left\lVert 
\bigl(x-y \bigr)-\bigl(\tilde{x}-\tilde{y} \bigr)\right\rVert
+
c\left[\sqrt{T}+p\right]\Biggl[\int_{s}^{t}
\Bigl\lVert 
\bigl(X_{s,\max\{s,\delta(r)\}}^{\funcPi,x}-X_{s,\max\{s,\delta(r)\}}^{\delta,y}\bigr)\nonumber\\[-12pt]
&\qquad\qquad\qquad\qquad\qquad\qquad\qquad\qquad\qquad\qquad
-
\bigl(
X_{s,\max\{s,\delta(r)\}}^{\funcPi,\tilde{x}}-X_{s,\max\{s,\delta(r)\}}^{\delta,\tilde{y}}\bigr)
\Bigr\rVert^2_{\frac{p}{2}}\,dr\Biggr]^{\nicefrac{1}{2}}\nonumber\\
& +
\left[\sqrt{T}+p\right]\biggl[
2(c^2+bc+b)
\left[\sqrt{T}+p\right]^3
e^{2 c^2\left[\sqrt{T}+p\right]^2 T}
e^{1.5 \barC T /p}
\frac{(V(x))^{\nicefrac{1}{p}}
+(V(\tilde{x}))^{\nicefrac{1}{p} }}{2}
\size{\delta}^{\nicefrac{1}{2}}\lVert x-\tilde{x} \rVert\nonumber \\
&\qquad\qquad\qquad\qquad\qquad + 2b e^{2c^2\left[
\sqrt{T}+p
\right]^2 T}
\frac{\left(\lVert x-y \rVert+
\lVert \tilde{x}-\tilde{y} \rVert\right) \lVert x-\tilde{x}\rVert}{2}
\biggr]\lvert t-s\rvert^{\nicefrac{1}{2}}
.\label{m23}
\end{align}
%
%
%
This, \cref{b02}, and integrability in \eqref{r06b} show for all 
$\delta\in\setSB$, $s\in[0,T]$,
$t\in[s,T]$, 
$x,\tilde{x},y,\tilde{y}\in\R^d$ that
\begin{align}
&\left\lVert 
\bigl(
X^{\funcPi,x}_{s,t}
-X^{\delta,y}_{s,t}
\bigr)
-\bigl(
X^{\funcPi,\tilde{x}}_{s,t}
-X^{\delta,\tilde{y}}_{s,t}
\bigr)
\right\rVert_{\frac{p}{2}}
 \leq 
\sqrt{2}e^{c^2\left[\sqrt{T}+p\right]^2T}\Biggl\{ \left\lVert 
\bigl(x-y \bigr)-\bigl(\tilde{x}-\tilde{y} \bigr)\right\rVert
\nonumber\\
& 
+
\left[\sqrt{T}+p\right]\biggl[
2(c^2+bc+b)
\left[\sqrt{T}+p\right]^3
e^{2 c^2\left[\sqrt{T}+p\right]^2 T}
e^{1.5 \barC T /p}
\frac{(V(x))^{\nicefrac{1}{p}}
+(V(\tilde{x}))^{\nicefrac{1}{p} }}{2}
\size{\delta}^{\nicefrac{1}{2}}\lVert x-\tilde{x} \rVert\nonumber \\
&\qquad\qquad\qquad\qquad\qquad + 2b e^{2c^2\left[
\sqrt{T}+p
\right]^2 T}
\frac{\left(\lVert x-y \rVert+
\lVert \tilde{x}-\tilde{y} \rVert\right) \lVert x-\tilde{x}\rVert}{2}
\biggr]\lvert t-s\rvert^{\nicefrac{1}{2}}\Biggr\}\nonumber\\
&
\leq \sqrt{2}e^{c^2\left[\sqrt{T}+p\right]^2T}
\left\lVert 
\bigl(
x-y \bigr)
-
\bigl(
\tilde{x}-\tilde{y} \bigr)\right\rVert\nonumber\\
&+
2\sqrt{2}(c^2+bc+b)
\left[\sqrt{T}+p\right]^4
e^{3 c^2\left[\sqrt{T}+p\right]^2 T}
e^{1.5 \barC T /p}
\frac{(V(x))^{\nicefrac{1}{p}}
+(V(\tilde{x}))^{\nicefrac{1}{p} }}{2}
\size{\delta}^{\nicefrac{1}{2}}\lVert x-\tilde{x} \rVert
\lvert t-s\rvert^{\nicefrac{1}{2}}
\nonumber \\
& + 2\sqrt{2}b\left[\sqrt{T}+p\right] e^{3 c^2\left[
\sqrt{T}+p
\right]^2 T}
\frac{\left(\lVert x-y \rVert+
\lVert \tilde{x}-\tilde{y} \rVert\right) \lVert x-\tilde{x}\rVert}{2}
\lvert t-s\rvert^{\nicefrac{1}{2}}.\label{m21}
\end{align}
This establishes \eqref{m33}.

Next, \eqref{y01} shows that for all 
$\delta\in \setS$, $x \in\R^d$,
$s\in[0,T]$, $\tilde{s}\in [s,T] $, $\bar{s}\in [\tilde{s},T]$
with $(s,\bar{s})\cap \delta ( [0,T] ) =\emptyset$
(cf. \cref{f01a})
it holds
that 
\begin{align}\begin{split}
X_{s,\bar{s}}^{\delta,x} -
X_{\tilde{s},\bar{s}}^{\delta,x}&= \bigl[ \mu(x)(\bar{s}-s) +\sigma (x)(W_{\bar{s}}-W_{s})\bigr]
-\bigl[
\mu(x)(\bar{s}-\tilde{s}) +\sigma (x)(W_{\bar{s}}-W_{\tilde{s}})\bigr]\\
&=\mu(x)(\tilde{s}-s)+\sigma(x)(W_{\tilde{s}}-W_{s})
\end{split}\end{align}
and
\begin{equation}\begin{split}
&
\bigl(
X_{s,\bar{s}}^{\funcPi,x}
-X_{s,\bar{s}}^{\delta,x} 
\bigr) -
\bigl(X_{\tilde{s},\bar{s}}^{\funcPi,x} -
X_{\tilde{s},\bar{s}}^{\delta,x}  
\bigr)=\bigl( X_{s,\bar{s}}^{\funcPi,x}
-X_{\tilde{s},\bar{s}}^{\funcPi,x} \bigr)
-\bigl(X_{s,\bar{s}}^{\delta,x} -
X_{\tilde{s},\bar{s}}^{\delta,x}  \bigr)\\
&=\left[
\int_{s}^{\bar{s}}
\mu(\procX_{s,r}^{\funcPi,x})\,dr
+\int_{s}^{\bar{s}}
\sigma(\procX_{s,r}^{\funcPi,x})\,dW_{r}\right]
-\left[
\int_{\tilde{s}}^{\bar{s}}
\mu(X_{\tilde{s},r}^{\funcPi,x})\,dr+
\int_{\tilde{s}}^{\bar{s}}
\sigma(X_{\tilde{s},r}^{\funcPi,x})\,dW_{r}\right]\\
&\qquad\qquad-\left[\int_{s}^{\tilde{s}}\mu(x)\,dr+\int_{s}^{\tilde{s}}\sigma(x)\,dW_{r}\right]\\
&=\left[\int_{s}^{\tilde{s}}\mu(\procX_{s,r}^{\funcPi,x})-\mu(x)\,dr
+\int_{s}^{\tilde{s}}
\sigma(\procX_{s,r}^{\funcPi,x})-\sigma(x)\,dW_{r}\right]\\
&\qquad\qquad+\left[\int_{\tilde{s}}^{\bar{s}}
\mu(\procX_{s,r}^{\funcPi,x})-\mu(X_{\tilde{s},r}^{\funcPi,x})\,dr
+\int_{\tilde{s}}^{\bar{s}}
\sigma(\procX_{s,r}^{\funcPi,x})-\sigma(X_{\tilde{s},r}^{\funcPi,x})\,dW_{r}\right].
\end{split}\label{m22}\end{equation}
This,
 the triangle inequality,
\eqref{m01}, 
\eqref{c03d}, 
and \eqref{m11} 
 show that for all $q\in [2,p/2]$, $\delta\in \setS$, $x \in\R^d$,
$s\in[0,T]$, $\tilde{s}\in [s,T] $, $\bar{s}\in [\tilde{s},T]$
with $(s,\bar{s})\cap \delta ( [0,T] ) =\emptyset$ (cf. \cref{f01a})
it holds
that
$\max\{\lvert \tilde{s}-s \rvert, \lvert\bar{s} -\tilde{s}\rvert\}\leq \size{\delta}$ and
\begin{align}\begin{split}
&
\left\lVert 
\bigl(X_{s,\bar{s}}^{\funcPi,x}-X_{s,\bar{s}}^{\delta,x}\bigr) -\bigl(X_{\tilde{s},\bar{s}}^{\funcPi,x} - X_{\tilde{s},\bar{s}}^{\delta,x}\bigr)
\right\rVert_{q}\\
&
\leq \xeqref{m22}\xeqref{m01}
c\left[
\sqrt{T}+q
\right]\size{\delta}^{\nicefrac{1}{2}}
\left[
\sup_{r\in [s,\tilde{s}]} \left\lVert X_{s,r}^{\funcPi,x}-x\right\rVert_{q}
+
\sup_{r\in \left[\tilde{s},\bar{s}\right]} \left\lVert X_{s,r}^{\funcPi,x}-X_{\tilde{s},r}^{\funcPi,x}
\right\rVert_{q}\right] \\
&\leq c\left[
\sqrt{T}+q
\right]\size{\delta}^{\nicefrac{1}{2}}\biggl[\xeqref{c03d}
\left[
\sqrt{T}+q
\right]
(e^{ 1.5 \barC T}V(x))^{\nicefrac{1}{p} }
\lvert s-\tilde{s}\rvert^{\nicefrac{1}{2}}
\\&\qquad\qquad\qquad\qquad\qquad+
\xeqref{m11}
{2}\sqrt{2}e^{c^2\left[
\sqrt{T}+q
\right]^2T}
\left[
\sqrt{T}+q
\right]
(e^{ 1.5 \barC T}V(x))^{\nicefrac{1}{p} }\lvert s-\tilde{s}\rvert^{\nicefrac{1}{2}}
\biggr].\end{split}
\end{align}
This and the fact that $1+2\sqrt{2}\leq 4$
show that for all $q\in [2,p/2]$, $\delta\in \setS$, $x \in\R^d$,
$s\in[0,T]$,
$\tilde{s}\in[s,T]$, 
 $\bar{s}\in [\tilde{s},T]$ 
with $(s,\bar{s})\cap \delta ( [0,T] ) =\emptyset$ (cf. \cref{f01a})
it holds
that
\begin{align}\begin{split}
&
\left\lVert 
\bigl(X_{s,\bar{s}}^{\funcPi,x}-X_{s,\bar{s}}^{\delta,x} \bigr) -\bigl(
X_{\tilde{s},\bar{s}}^{\funcPi,x} -X_{\tilde{s},\bar{s}}^{\delta,x}\bigr)
\right\rVert_{q}\\
&\leq 
{4}c\left[
\sqrt{T}+q
\right]^2
e^{c^2\left[
\sqrt{T}+q
\right]^2T}
(e^{ 1.5 \barC T}V(x))^{\nicefrac{1}{p} }\lvert s-\tilde{s}\rvert^{\nicefrac{1}{2}}
\size{\delta}^{\nicefrac{1}{2}}
.
\end{split}\label{m30}\end{align}
This, the fact  that the Euler approximations 
in \eqref{y01}
restricted to their grid points satisfy the Markov property,
the disintegration theorem (see, e.g., \cite[Lemma 2.2]{HJK+18}),
\eqref{m21}, 
 the fact that $\forall\,t,s\in[0,T]\colon \lvert t-s\rvert^{\nicefrac{1}{2}}\leq \sqrt{T}$, the triangle inequality,
H\"older's inequality, the fact that $p\geq 4$,
\eqref{r06}, 
\eqref{m11}, 
\eqref{m12}, 
 the fact that $V\geq 1$,
and the fact that
$\sqrt{2}\cdot 4c+ 2\sqrt{2} (c^2+bc+b) 2\sqrt{2}+\sqrt{2}b\cdot 2\sqrt{2}c\cdot 2\sqrt{2}
\leq 6c+8(c^2+bc+b)+12 bc
\leq 20(c^2+bc+b+c)=20(b+c)(c+1)
$
show that 
for all  $\delta\in \setS$, $x \in\R^d$,
$s\in[0,T]$, 
$\tilde{s}\in [s,T]$,
$t\in [\tilde{s},T]$,
$\bar{s}\in[\tilde{s},t]\cap\bigl(\delta([0,T])\cup \{t\}\bigr)$
with 
$(s,\bar{s})\cap \delta([0,T]) =\emptyset$ (cf. \cref{f01a,f01b})
it holds
that
\begin{align}
&\left\lVert 
\bigl(
X_{s,t}^{\funcPi,x}-
X_{s,t}^{\delta,x} 
\bigr) -
\bigl(X_{\tilde{s},t}^{\funcPi,x}-X_{\tilde{s},t}^{\delta,x} \bigr)\right\rVert_{\frac{p}{2}}\nonumber \\
&=\left\lVert \left\lVert 
\bigl(
X_{\bar{s},t}^{\funcPi,\mathbf{x}}-X_{\bar{s},t}^{\delta,\mathbf{y}}
\bigr)
 -
\bigl(
X_{\bar{s},t}^{\funcPi,\mathbf{\tilde{x}}}
-X_{\bar{s},t}^{\delta,\mathbf{\tilde{y}}} \bigr)\right\rVert_{\frac{p}{2}}\bigr|_{
\substack{
\mathbf{x}=X_{s,\bar{s}}^{\funcPi,x},
\mathbf{y} =X_{s,\bar{s}}^{\delta,x},
\mathbf{\tilde{x}}=
X^{\funcPi,x}_{\tilde{s},\bar{s}},
\mathbf{\tilde{y}}=X_{\tilde{s},\bar{s}}^{\delta,x}
}}
\right\rVert_{\frac{p}{2}}\nonumber \\
&\leq
\xeqref{m21}
\Biggl \lVert  \Biggl[\sqrt{2}e^{c^2\left[\sqrt{T}+p\right]^2 T}
\left\lVert 
\bigl(
\mathbf{x}-\mathbf{y} \bigr)
-
\bigl(
\tilde{\mathbf{x}}-\tilde{\mathbf{y}} \bigr)\right\rVert\nonumber\\
&\quad+2\sqrt{2}(c^2+bc+b)
\left[\sqrt{T}+p\right]^5
e^{3c^2\left[\sqrt{T}+p\right]^2 T}
e^{1.5\barC T /p}
\frac{(V(\mathbf{x}))^{\nicefrac{1}{p}}
+(V(\tilde{\mathbf{x}}))^{\nicefrac{1}{p} }}{2}
\size{\delta}^{\nicefrac{1}{2}}\lVert \mathbf{x}-\tilde{\mathbf{x}} \rVert
\nonumber \\
&\quad  +2\sqrt{2}b\left[\sqrt{T}+p\right]^2 e^{3c^2\left[
\sqrt{T}+p
\right]^2 T}
\frac{\left(\lVert \mathbf{x}-\mathbf{y} \rVert+
\lVert \tilde{\mathbf{x}}-\tilde{\mathbf{y}} \rVert\right) \lVert \mathbf{x}-\tilde{\mathbf{x}}\rVert}{2}
\Biggr]\Bigr|_{
\substack{
\mathbf{x}=X_{s,\bar{s}}^{\funcPi,x},
\mathbf{y} =X_{s,\bar{s}}^{\delta,x},
\mathbf{\tilde{x}}=
X^{\funcPi,x}_{\tilde{s},\bar{s}},
\mathbf{\tilde{y}}=X_{\tilde{s},\bar{s}}^{\delta,x}
}}
\Biggr \rVert_{\frac{p}{2}}\nonumber \\
&\leq 
\sqrt{2}e^{c^2\left[\sqrt{T}+p\right]^2 T}
\left\lVert 
\bigl(
X_{s,\bar{s}}^{\funcPi,x}-X_{s,\bar{s}}^{\delta,x}
\bigr)-
\bigl(
X^{\funcPi,x}_{\tilde{s},\bar{s}}-X_{\tilde{s},\bar{s}}^{\delta,x}
\bigr)\right\rVert_{\frac{p}{2}}\nonumber\\
&\quad+2\sqrt{2}(c^2+bc+b)
\left[\sqrt{T}+p\right]^5
e^{3c^2\left[\sqrt{T}+p\right]^2 T}\nonumber\\
&\qquad\qquad
\left\lVert 
e^{ 1.5 \barC T/p}\frac{
(V(X_{s,\bar{s}}^{\funcPi,x}))^{\nicefrac{1}{p}}+(V(X_{\tilde{s},\bar{s}}^{\funcPi,x}))^{\nicefrac{1}{p}}}{2}
\right\rVert_{p}
\size{\delta}^{\nicefrac{1}{2}}\left\lVert X_{s,\bar{s}}^{\funcPi,x}-X_{\tilde{s},\bar{s}}^{\funcPi,x}\right\rVert_{p}
\nonumber \\
&\quad  +\sqrt{2}b\left[\sqrt{T}+p\right]^2 e^{3c^2\left[
\sqrt{T}+p
\right]^2 T}
\left(
\left\lVert 
X_{s,\bar{s}}^{\funcPi,x} -X_{s,\bar{s}}^{\delta,x}
\right\rVert_{p}
+
\left\lVert X_{\tilde{s},\bar{s}}^{\funcPi,x}-X_{\tilde{s},\bar{s}}^{\delta,x} \right\rVert_{p}
\right) \left\lVert X_{s,\bar{s}}^{\funcPi,x}-X_{\tilde{s},\bar{s}}^{\funcPi,x}\right\rVert_{p}
\nonumber\\
&\leq 
\sqrt{2}e^{c^2\left[\sqrt{T}+p\right]^2 T}\xeqref{m30}
4 c\left[\sqrt{T}+p\right]^2
e^{c^2\left[
\sqrt{T}+p
\right]^2T}
(e^{ 1.5 \barC T}V(x))^{\nicefrac{1}{p} }\lvert s-\tilde{s}\rvert^{\nicefrac{1}{2}}
\size{\delta}^{\nicefrac{1}{2}}\nonumber\\
&\quad+ 2\sqrt{2}(c^2+bc+b)
\left[\sqrt{T}+p\right]^5
e^{3c^2\left[\sqrt{T}+p\right]^2 T}\nonumber\\
&\qquad\qquad\xeqref{r06}
e^{ 3 \barC T/p}
(V(x))^{\nicefrac{1}{p}}
\size{\delta}^{\nicefrac{1}{2}}\xeqref{m11}
{2}\sqrt{2}e^{c^2\left[
\sqrt{T}+p
\right]^2T}
\left[
\sqrt{T}+p
\right]
(e^{1.5 \barC T}V(x))^{\nicefrac{1}{p} }\lvert s-\tilde{s}\rvert^{\nicefrac{1}{2}}
\nonumber \\
&\quad  +\sqrt{2}b\left[\sqrt{T}+p\right]^2 e^{3c^2\left[
\sqrt{T}+p
\right]^2T}
2
\xeqref{m12}
\sqrt{2}e^{ c^2\left[
\sqrt{T}+p
\right]^2 T}c\left[\sqrt{T}+p\right]^3 
(e^{1.5 \barC T}V(x))^{\nicefrac{1}{p} }
 \size{\delta}^{\nicefrac{1}{2}}\nonumber\\
&\qquad\qquad \xeqref{m11} {2}\sqrt{2}e^{c^2\left[
\sqrt{T}+p
\right]^2T}
\left[
\sqrt{T}+p
\right]
(e^{1.5 \barC T}V(x))^{\nicefrac{1}{p} }\lvert s-\tilde{s}\rvert^{\nicefrac{1}{2}} \nonumber\\
&\leq 20(b+c)(c+1)e^{5 c^2\left[\sqrt{T}+p\right]^2 T}
\left[\sqrt{T}+p\right]^6
e^{ 4.5 \barC T/p}
(V(x))^{\nicefrac{2}{p}}
\lvert s-\tilde{s}\rvert^{\nicefrac{1}{2}}
 \size{\delta}^{\nicefrac{1}{2}}
.\label{m25b}
\end{align}
Furthermore, note that (cf. \cref{f01a,f01b}) for all $\delta\in \setS$,
$s\in [0,T]$, 
$\tilde{s}\in[s,T]$,
$t\in [\tilde{s},T]$
with $ (s,\tilde{s})\cap \delta([0,T])=\emptyset$
 there exists $\bar{s}\in[\tilde{s},t]\cap\bigl(\delta([0,T])\cup \{t\}\bigr)$
with 
$(s,\bar{s})\cap \delta([0,T]) =\emptyset$.
This and \eqref{m25b} show
for all $\delta\in \setS$,
$s\in [0,T]$, 
$\tilde{s}\in[s,T]$,
$t\in [\tilde{s},T]$ that
\begin{align}
&\left\lVert 
\bigl(
X_{s,t}^{\funcPi,x}-
X_{s,t}^{\delta,x} 
\bigr) -
\bigl(X_{\tilde{s},t}^{\funcPi,x}-X_{\tilde{s},t}^{\delta,x} \bigr)\right\rVert_{\frac{p}{2}}\nonumber \\
&\leq 20(b+c)(c+1)e^{5 c^2\left[\sqrt{T}+p\right]^2 T}
\left[\sqrt{T}+p\right]^6
e^{ 4.5 \barC T/p}
(V(x))^{\nicefrac{2}{p}}
\lvert s-\tilde{s}\rvert^{\nicefrac{1}{2}}
 \size{\delta}^{\nicefrac{1}{2}}
.\label{m25}
\end{align}
Next, 
the fact  that the Euler approximations 
in \eqref{y01}
restricted to their grid points satisfy the Markov property,
the disintegration theorem (see, e.g., \cite[Lemma 2.2]{HJK+18}),
\eqref{m21}, 
the fact that $\forall\,t,s\in[0,T]\colon \lvert t-s\rvert^{\nicefrac{1}{2}}\leq \sqrt{T}$, the triangle inequality,
H\"older's inequality, the fact that $p\geq 4$,
 \eqref{m12}, 
\eqref{r06}, 
 \eqref{c03d}, 
the fact that $V\geq 1$,
and the fact that
$\sqrt{2}\cdot\sqrt{2}c+2\sqrt{2}(c^2+bc+b)+\sqrt{2}b\sqrt{2}c \leq 2c+3(c^2+bc+b)+2bc\leq 5(c^2+bc+b+c)=5(b+c)(c+1)$
 show 
 for all $\delta\in\setS$, $s\in [0,T]$, $\tilde{s}\in \delta([0,T])\cap[s,T]$, $t\in [\tilde{s},T]$ (cf. \cref{f01c}) that 
\begin{align}
&\left\lVert 
\bigl(X_{s,t}^{\funcPi,x}
-X_{s,t}^{\delta,x} 
\bigr)
-\bigl(
X_{\tilde{s},t}^{\funcPi,x} -
X_{\tilde{s},t}^{\delta,x} 
\bigr)
\right\rVert_{\frac{p}{2}}=
\left\lVert \left \lVert 
\bigl(X_{\tilde{s},t}^{\funcPi,\mathbf{x}}-X_{\tilde{s},t}^{\delta,\mathbf{y}}\bigr) -
\bigl(X_{\tilde{s},t}^{\funcPi,x} -X_{\tilde{s},t}^{\delta,x}
\bigr)
\right\rVert_{\frac{p}{2}}
\bigr|_{
\mathbf{x}= X^{\funcPi,x}_{s,\tilde{s}},\mathbf{y}= X^{\delta,x}_{s,\tilde{s}}
}
\right\rVert_{\frac{p}{2}}\nonumber \\
&\leq\xeqref{m21} \Biggl \lVert \biggl[\sqrt{2}e^{c^2\left[\sqrt{T}+p\right]^2T}
\left\lVert 
\bigl(
\mathbf{x}-\mathbf{y} \bigr)
-
\bigl({x}-{x} \bigr)\right\rVert\nonumber\\
&\quad+ 2\sqrt{2}(c^2+bc+b)
\left[\sqrt{T}+p\right]^5
e^{3 c^2\left[\sqrt{T}+p\right]^2 T}
e^{1.5 \barC T /p}
\frac{(V(\mathbf{x}))^{\nicefrac{1}{p}}
+(V({x}))^{\nicefrac{1}{p} }}{2}
\size{\delta}^{\nicefrac{1}{2}}\lVert \mathbf{x}-{x} \rVert
\nonumber \\
&\quad  +2\sqrt{2}b\left[\sqrt{T}+p\right]^2 
e^{3 c^2\left[\sqrt{T}+p\right]^2 T}
\frac{\left(\lVert \mathbf{x}-\mathbf{y} \rVert+
\lVert {x}-{x} \rVert\right) \lVert \mathbf{x}-{x}\rVert}{2}
\biggr]
\bigr|_{
\mathbf{x}= X^{\funcPi,x}_{s,\tilde{s}},\mathbf{y}= X^{\delta,x}_{s,\tilde{s}}
 }\Biggr \rVert_{{\frac{p}{2}}}\nonumber \\
&\leq \sqrt{2}e^{c^2\left[\sqrt{T}+p\right]^2 T}
\left\lVert 
X^{\funcPi,x}_{s,\tilde{s}}- X^{\delta,x}_{s,\tilde{s}} 
\right\rVert_{\frac{p}{2}}\nonumber\\
&\quad+2\sqrt{2}(c^2+bc+b)\nonumber\\
&\qquad\qquad
\left[\sqrt{T}+p\right]^5
e^{3 c^2\left[\sqrt{T}+p\right]^2 T}
e^{1.5\barC T /p}
\left\lVert\frac{(V(X^{\funcPi,x}_{s,\tilde{s}}))^{\nicefrac{1}{p}}
+(V({x}))^{\nicefrac{1}{p} }}{2}\right\rVert_{p}
\size{\delta}^{\nicefrac{1}{2}}\left\lVert X^{\funcPi,x}_{s,\tilde{s}}-{x} \right\rVert_{p}
\nonumber \\
&\quad  + \sqrt{2}b\left[\sqrt{T}+p\right]^2 e^{3c^2\left[
\sqrt{T}+p
\right]^2 T}
\left\lVert X^{\funcPi,x}_{s,\tilde{s}}-X^{\delta,x}_{s,\tilde{s}} \right\rVert_{p}
\left \lVert X^{\funcPi,x}_{s,\tilde{s}}-{x}\right\rVert_{p}\nonumber\\
&\leq \sqrt{2}e^{c^2\left[\sqrt{T}+p\right]^2 T}
\xeqref{m12}
 \sqrt{2}e^{ c^2\left[\sqrt{T}+p\right]^2 T}c\left[\sqrt{T}+p\right]^2
(e^{1.5 \barC T}V(x))^{\nicefrac{1}{p} }
\lvert s-\tilde{s}\rvert^{\nicefrac{1}{2}} \size{\delta}^{\nicefrac{1}{2}}\nonumber\\
&\quad+ 2\sqrt{2}(c^2+bc+b)
\left[\sqrt{T}+p\right]^5
e^{3c^2\left[\sqrt{T}+p\right]^2 T}\xeqref{r06}
e^{3 \barC T /p}(V({x}))^{\nicefrac{1}{p} }
\size{\delta}^{\nicefrac{1}{2}}\nonumber\\
&\qquad\qquad\xeqref{c03d}
\left[\sqrt{T}+p\right]
(e^{1.5\barC T}V(x))^{\nicefrac{1}{p} }\lvert s-\tilde{s}\rvert^{\nicefrac{1}{2}}
\nonumber \\
&\quad +\sqrt{2}b
\left[\sqrt{T}+p\right]^2
e^{3c^2\left[\sqrt{T}+p\right]^2 T}\xeqref{m12} 
\sqrt{2}e^{ c^2\left[\sqrt{T}+p\right]^2 T}c\left[\sqrt{T}+p\right]^3
(e^{1.5 \barC T}V(x))^{\nicefrac{1}{p} }
 \size{\delta}^{\nicefrac{1}{2}}
\nonumber\\
&\qquad\qquad
\xeqref{c03d}\left[
\sqrt{T}+p
\right]
(e^{1.5\barC T}V(x))^{\nicefrac{1}{p} }\lvert s-\tilde{s}\rvert^{\nicefrac{1}{2}}\nonumber\\
&\leq 5(b+c)(c+1)\left[\sqrt{T}+p\right]^6
e^{4c^2\left[\sqrt{T}+p\right]^2 T}
e^{4.5 \barC T/p}(V(x))^{\nicefrac{2}{p} }
\lvert {s}-\tilde{s}\rvert^{\nicefrac{1}{2}} \size{\delta}^{\nicefrac{1}{2}}.
\label{m34}\end{align}
This, the triangle inequality, and \eqref{m25} show that for all
$\delta\in\setS$,
$s\in [0,T]$, 
$\tilde{s}\in[s,T]$,
$t\in [\tilde{s},T]$, $\bar{s}\in [s,\tilde{s}]\cap (\delta([0,T]) \cup\{s\})$ with 
$ (\bar{s},\tilde{s})\cap\delta([0,T])=\emptyset$ (cf. \cref{f01d}) it holds that
$\max\{ \lvert s-\bar{s}\rvert, \lvert \bar{s}-\tilde{s}\rvert\}\leq \lvert s-\tilde{s}\rvert $ and
\begin{align}\begin{split}
&\left \lVert \bigl(X_{\tilde{s},t}^{\funcPi,x} -
X_{\tilde{s},t}^{\delta,x}
\bigr)
-\bigl(X_{s,t}^{\funcPi,x}-X_{s,t}^{\delta,x}\bigr)\right\rVert_{\frac{p}{2}}
\\
&
\leq \left\lVert 
\bigl(
X_{\tilde{s},t}^{\funcPi,x} -X_{\tilde{s},t}^{\delta,x}
\bigr)
-
\bigl(X_{\bar{s},t}^{\funcPi,x}-X_{\bar{s},t}^{\delta,x}\bigr)\right\rVert_{\frac{p}{2}}
+\left\lVert \bigl(X_{\bar{s},t}^{\funcPi,x}-X_{\bar{s},t}^{\delta,x}\bigr)
-\bigl(X_{s,t}^{\funcPi,x}-X_{s,t}^{\delta,x}\bigr)
\right\rVert_{\frac{p}{2}}  \\
&\leq \xeqref{m25}20 (b+c)(c+1)e^{5 c^2\left[\sqrt{T}+p\right]^2 T}
\left[\sqrt{T}+p\right]^6
e^{ 4.5 \barC T/p}
(V(x))^{\nicefrac{2}{p}}
\lvert \tilde{s}-\bar{s}\rvert^{\nicefrac{1}{2}}
 \size{\delta}^{\nicefrac{1}{2}}\\
&\quad +\xeqref{m34}5 (b+c)(c+1)\left[\sqrt{T}+p\right]^6
e^{4c^2\left[\sqrt{T}+p\right]^2 T}
e^{4.5 \barC T/p}(V(x))^{\nicefrac{2}{p} }
\lvert \bar{s}-s\rvert^{\nicefrac{1}{2}} \size{\delta}^{\nicefrac{1}{2}}\\
&\leq
25 (b+c)(c+1)\left[\sqrt{T}+p\right]^6
e^{5c^2\left[\sqrt{T}+p\right]^2 T}
e^{4.5 \barC T/p}(V(x))^{\nicefrac{2}{p} }
\lvert s-\tilde{s}\rvert^{\nicefrac{1}{2}} \size{\delta}^{\nicefrac{1}{2}}.
\end{split}\label{m10b}\end{align}
Next, note that (cf. \cref{f01d}) for all $\delta\in \setS$,
$s\in [0,T]$, 
$\tilde{s}\in[s,T]$,
$t\in [\tilde{s},T]$ there exists $\bar{s}\in [s,\tilde{s}]\cap (\delta([0,T]) \cup\{s\})$ with 
$ (\bar{s},\tilde{s})\cap\delta([0,T])=\emptyset$.
This, \eqref{m10b}, and symmetry imply 
for all $\delta\in\setS$,
$s,\tilde{s}\in[0,T]$, $t\in [\max\{s,\tilde{s}\},T]$
that
\begin{align}\begin{split}
&
\left \lVert \bigl(
X_{\tilde{s},t}^{\funcPi,x} -
X_{\tilde{s},t}^{\delta,x}
\bigr)
-\bigl(X_{s,t}^{\funcPi,x}-X_{s,t}^{\delta,x} \bigr)\right\rVert_{\frac{p}{2}}
\\
&
\leq 
25 (b+c)(c+1)  \left[\sqrt{T}+p\right]^6 e^{5 c^2\left[\sqrt{T}+p\right]^2 T}
e^{ 4.5 \barC T/p}(V(x))^{\nicefrac{2}{p}}
\lvert s-\tilde{s}\rvert^{\nicefrac{1}{2}}\size{\delta}^{\nicefrac{1}{2}}.
\label{m10}
\end{split}\end{align}
Next, \eqref{y01}, 
 the triangle inequality,  
\eqref{m01}, 
\eqref{c03d}, 
and \eqref{m12} 
prove for all $\delta\in \setS$,
 $ s\in[0,T]$, $t\in[s,T]$, 
$\tilde{t}\in[t,T]$,
$x\in\R^d$ that
\begin{align}
&\left\lVert 
\bigl(X^{\funcPi,x}_{s,\tilde{t}}-
X_{s,\tilde{t}}^{\delta,x}\bigr)-\bigl(X^{\funcPi,x}_{s,t}-X_{s,t}^{\delta,x}\bigr)\right\rVert_{\frac{p}{2}}\nonumber \\
&
= \xeqref{y01} \left\lVert \int_{t}^{\tilde{t}}
\mu(X^{\funcPi,x}_{s,r})-\mu(X^{\delta,x}_{s,\max\{s,\rdown{r}{\delta}\}})
\,dr
+\int_{t}^{\tilde{t}}
\sigma(X^{\funcPi,x}_{s,r})
-\sigma(X^{\delta,x}_{s,\max\{s,\rdown{r}{\delta}\}})
\,dW_{r}\right\rVert_{\frac{p}{2}}\nonumber \\
&\leq \xeqref{m01}
 c\left[\sqrt{T}+p\right] \lvert t-\tilde{t}\rvert^{\nicefrac{1}{2}} \sup_{r\in\left[t,\tilde{t}\right]}
\left \lVert
X^{\funcPi,x}_{s,r}
-X^{\delta,x}_{s,\max\{s,\rdown{r}{\delta}\}}
\right\rVert_{\frac{p}{2}}\nonumber \\
&\leq
  c\left[\sqrt{T}+p\right] \lvert t-\tilde{t}\rvert^{\nicefrac{1}{2}} \nonumber
\\&\qquad\qquad
\sup_{r\in\left[t,\tilde{t}\right]} \left[
\left \lVert
X^{\funcPi,x}_{s,r}
-X^{\funcPi,x}_{s,\max\{s,\rdown{r}{\delta}\}}
\right\rVert_{\frac{p}{2}}
+\left\lVert 
X^{\funcPi,x}_{s,\max\{s,\rdown{r}{\delta}\}}
-
X^{\delta,x}_{s,\max\{s,\rdown{r}{\delta}\}}\right\rVert_{\frac{p}{2}}
\right]\nonumber 
 \\
&\leq 
 c\left[\sqrt{T}+p\right]\lvert t-\tilde{t}\rvert^{\nicefrac{1}{2}} 
\biggl[\xeqref{c03d}\left[
\sqrt{T}+p
\right]
(e^{1.5 \barC T}V(x))^{\nicefrac{1}{p} }\size{\delta}^{\nicefrac{1}{2}}\nonumber\\
&\qquad\qquad\qquad\qquad\qquad\qquad
+\xeqref{m12}
\sqrt{2}e^{ c^2\left[\sqrt{T}+p\right]^2 T}
c\left[\sqrt{T}+p\right]^3
(e^{1.5 \barC T}V(x))^{\nicefrac{1}{p} }
 \size{\delta}^{\nicefrac{1}{2}}\biggr]
\nonumber \\
&\leq \left[\sqrt{2}c^2+c\right]
\left[\sqrt{T}+p\right]^4
e^{ c^2\left[\sqrt{T}+p\right]^2 T}
(e^{1.5 \barC T}V(x))^{\nicefrac{1}{p} }\size{\delta}^{\nicefrac{1}{2}} \lvert t-\tilde{t}\rvert^{\nicefrac{1}{2}}.\label{m35}
\end{align}
This, symmetry, the triangle inequality, 
\eqref{m10}, 
\eqref{m21} (applied with 
$(x,y,\tilde{x},\tilde{y})\gets (x,x,\tilde{x},\tilde{x}) $ in the notation of
\eqref{m21}), 
the fact that
$\forall\,t,s\in[0,T]\colon \lvert t-s\rvert^{\nicefrac{1}{2}}\leq \sqrt{T}+p$,
 \eqref{m31},
the fact that $V\geq 1$,
and the fact that
$2(\sqrt{2}c^2+c)+25 (b+c)(c+1)
+2\sqrt{2}(c^2+bc+b)
\leq 25(b+c)(c+1)+6(c^2+bc+b+c)= 31 (b+c)(c+1)$
prove that for all $\delta\in \setS$,
 $s,\tilde{s},t,\tilde{t}\in[0,T]$, $x,\tilde{x}\in\R^d$ with 
$V(x)\leq V(\tilde{x})$ it holds that
\begin{align}
&
\left\lVert \bigl(
X_{s,\max\{s,t\}}^{\funcPi,x}
-X_{s,\max\{s,t\} }^{\delta,x} 
\bigr)
 -
\bigl(
X_{\tilde{s},\max\{\tilde{s},\tilde{t}\}}^{\funcPi,\tilde{x}}
-X_{\tilde{s},\max\{\tilde{s},\tilde{t}\}}^{\delta,\tilde{x}} \bigr)\right\rVert_{{\frac{p}{2}}} \nonumber\\
&
\leq \left\lVert \bigl(X_{s,\max\{s,t\}}^{\funcPi,x}
-X_{s,\max\{s,t\} }^{\delta,x}
\bigr)
-
\bigl(
X_{s,\max\{s,\tilde{s},t,\tilde{t}\} }^{\funcPi,x}-
X_{s,\max\{s,\tilde{s},t,\tilde{t}\} }^{\delta,x}
\bigr)\right\rVert_{{\frac{p}{2}}} \nonumber\\
&\quad +\left\lVert\bigl(
X_{s,\max\{s,\tilde{s},t,\tilde{t}\} }^{\funcPi,x}
-X_{s,\max\{s,\tilde{s},t,\tilde{t}\} }^{\delta,x}
\bigr)
-\bigl(
 X_{\tilde{s},\max\{s,\tilde{s},t,\tilde{t}\} }^{\funcPi,x}
-X_{\tilde{s},\max\{s,\tilde{s},t,\tilde{t}\} }^{\delta,x}
\bigr)\right\rVert_{{\frac{p}{2}}} \nonumber\\
&\quad +\left\lVert\bigl(
X_{\tilde{s},\max\{s,\tilde{s},t,\tilde{t}\} }^{\funcPi,x}
-X_{\tilde{s},\max\{s,\tilde{s},t,\tilde{t}\} }^{\delta,x}
\bigr)
-\bigl( X_{\tilde{s},\max\{\tilde{s},\tilde{t}\}}^{\funcPi,x}
-X_{\tilde{s},\max\{\tilde{s},\tilde{t}\}}^{\delta,x} \bigr)\right\rVert_{{\frac{p}{2}}} \nonumber\\
&\quad +\left\lVert\bigl(
X_{\tilde{s},\max\{\tilde{s},\tilde{t}\}}^{\funcPi,x} 
-X_{\tilde{s},\max\{\tilde{s},\tilde{t}\}}^{\delta,x}
\bigr)-
\bigl(X_{\tilde{s},\max\{\tilde{s},\tilde{t}\}}^{\funcPi,\tilde{x}}
-X_{\tilde{s},\max\{\tilde{s},\tilde{t}\}}^{\delta,\tilde{x}} \bigr)\right\rVert_{{\frac{p}{2}}} \nonumber\\
&\leq \xeqref{m35}
\left[\sqrt{2} c^2+c\right]
\left[\sqrt{T}+p\right]^4
e^{ c^2\left[\sqrt{T}+{p}\right]^2 T}
(e^{1.5\barC T}V(x))^{\nicefrac{1}{p} } \size{\delta}^{\nicefrac{1}{2}}
\left\lvert\max\{s,t\}-\max\{s,\tilde{s},t,\tilde{t}\} \right \rvert^{\nicefrac{1}{2}} \nonumber\\
&\quad + \xeqref{m10}
25 (b+c)(c+1)    \left[\sqrt{T}+p\right]^6 e^{5 c^2\left[\sqrt{T}+p\right]^2 T}
e^{ 4.5 \barC T/p}(V(x))^{\nicefrac{2}{p}}\size{\delta}^{\nicefrac{1}{2}}
\lvert s-\tilde{s}\rvert^{\nicefrac{1}{2}}
 \nonumber\\
&\quad +
\xeqref{m35}\left[\sqrt{2} c^2+c\right]
\left[\sqrt{T}+p\right]^4
e^{ c^2\left[\sqrt{T}+{p}\right]^2 T}
(e^{1.5\barC T}V(x))^{\nicefrac{1}{p} } \size{\delta}^{\nicefrac{1}{2}}
\left\lvert\max\{s,\tilde{s},t,\tilde{t}\}-\max\{\tilde{s},\tilde{t}\}\right\rvert^{\nicefrac{1}{2}}
\nonumber\\
&\quad + \xeqref{m21}
2\sqrt{2} (c^2+bc+b) 
\left[\sqrt{T}+p\right]^5e^{3 c^2\left[\sqrt{T}+p\right]^2 T}
e^{1.5 \barC T/p}
\frac{(V(x))^{\nicefrac{1}{p}}
+(V(\tilde{x}))^{\nicefrac{1}{p} }}{2}
\size{\delta}^{\nicefrac{1}{2}}\lVert x-\tilde{x} \rVert\nonumber\\
&\leq 31 (b+c)(c+1)
\left[\sqrt{T}+p\right]^6
e^{5 c^2\left[\sqrt{T}+p\right]^2 T}\nonumber\\
&\qquad\qquad
 e^{4.5\barC T/p}\frac{(V(x))^{\nicefrac{2}{p}}
+(V(\tilde{x}))^{\nicefrac{2}{p} }}{2}
\left[
\lvert s-\tilde{s}\rvert^{\nicefrac{1}{2}}+
\lvert t-\tilde{t}\rvert^{\nicefrac{1}{2}}+\lVert x-\tilde{x}\rVert\right]
\size{\delta}^{\nicefrac{1}{2}}.
\end{align}
This and symmetry
imply \eqref{m32}.
The proof of \cref{s01b} is thus completed.
\end{proof}
\begin{lemma}\label{f01}Let $(V,\lVert \cdot \rVert_V)$, 
$(W,\lVert \cdot \rVert_W)$ be finite-dimensional normed $\R$-vector spaces with $V\neq \{0\}$ and let 
$f\in C^2(V,W)$.
Then
it holds for all $v_1,v_2,w_1,w_2\in V $
 that
\begin{align}
&\lVert (f(v_1)-f(w_1)) - (f(v_2) -f(w_2))\rVert_W\leq \left[\sup_{v\in V,h\in V\setminus\{0\}}
\frac{\left\lVert (\totalD f(v))(h)\right\rVert_W}{\lVert h\rVert_V}\right]\lVert(v_1-w_1) -(v_2-w_2)\rVert_V\nonumber\\
&\quad +\frac{1}{2}\left[\sup_{v\in V, h,k\in \setminus\{0\}}\frac{\left\lVert (\totalD^2 f (v))(h,k)\right\rVert_W}{\lVert h \rVert_V\lVert k\rVert_V}\right]\Bigl[\lVert v_1-v_2\rVert_V+\lVert w_1-w_2\rVert_V\Bigr]\lVert w_1-w_2\rVert_V.
\nonumber\end{align}
\end{lemma}
\begin{proof}[Proof of \cref{f01}]  
The 
 fundamental theorem of calculus and the triangle inequality
 show for all
$v_1,v_2,w_1,w_2\in V $ that
\begin{align}\label{c06c}
&\left\lVert(f(v_1)-f(w_1)) - (f(v_2) -f(w_2))\right\rVert_W=
\left\lVert(f(v_1)-f(v_2)) - (f(w_1) -f(w_2))\right\rVert_W
   \nonumber\\
&
= \left\lVert \int_{0}^{1} ( \totalD f(\lambda v_1+(1-\lambda)v_2 ))(v_1-v_2) -
 \totalD f(\lambda w_1+(1-\lambda)w_2 )(w_1-w_2)\,d\lambda\right\rVert_W   \nonumber\\
&= \biggl\lVert\int_{0}^{1} \bigl[\totalD f(\lambda v_1+(1-\lambda)v_2 )\bigr]((v_1-v_2)-(w_1-w_2) )\,d\lambda  \nonumber\\
&\qquad +\int_{0}^{1}\bigl[
\totalD f(\lambda v_1+(1-\lambda)v_2 )-
\totalD f(\lambda w_1+(1-\lambda)w_2 )\bigr](w_1-w_2)\,d\lambda \biggr\rVert_W   \nonumber\\
&= \biggl\lVert\int_{0}^{1} \bigl[\totalD f(\lambda v_1+(1-\lambda)v_2 )\bigr]((v_1-v_2)-(w_1-w_2) )\,d\lambda  \nonumber\\
&\qquad +\int_{0}^{1}\int_{0}^{1}\left[
\totalD^2f\Bigl( \mu (\lambda v_1+(1-\lambda)v_2)+(1-\mu)(\lambda w_1+(1-\lambda)w_2 ) \Bigr)
\right]\nonumber\\
&\qquad\qquad\qquad\qquad\qquad\qquad\qquad\Bigl(\lambda (v_1-w_1)+(1-\lambda)(v_2-w_2),w_1-w_2\Bigr)\,d\mu \,d\lambda \biggr\rVert_W \nonumber\\
&\leq \left[\sup_{v\in V,h\in V\setminus\{0\}}
\frac{\left\lVert (\totalD f(v))(h)\right\rVert_W}{\lVert h\rVert_V}\right]\lVert (v_1-w_1) -(v_2-w_2)\rVert_V  \nonumber\\
&\quad +\left[\sup_{v\in V, h,k\in \setminus\{0\}}\frac{\left\lVert (\totalD^2 f (v))(h,k)\right\rVert_W}{\lVert h \rVert_V\lVert k\rVert_V}\right]\!\left[\int_{0}^{1}
\lambda\lVert v_1-w_1\rVert_V+(1-\lambda)\lVert v_2-w_2\rVert_V\,d\lambda\right]\!\lVert w_1-w_2\rVert_V.
\end{align}%
This and the fact that $\int_{0}^{1}\lambda\,d\lambda=
\int_{0}^{1}(1-\lambda)\,d\lambda=1/2
$ complete the proof of \cref{f01}.
\end{proof}
\begin{corollary}\label{c10}
Let $\lVert \cdot\rVert\colon \bigcup_{m,n\in\N}\R^{m\times n}\to[0,\infty)$ satisfy for all $m,n\in\N$, $s=(s_{ij})_{i\in[1,m]\cap\N,j\in [1,n]\cap\N}\in\R^{m\times n}$ that
$\lVert s\rVert^2=\sum_{i=1}^{m}\sum_{j=1}^{n}\lvert s_{ij}\rvert^2$,
let $d\in\N$, 
$T,p\in (0,\infty)$, 
let 
$\mu\in C^2(\R^d,\R^d) $, $ \sigma  \in C^2(\R^d,\R^{d\times d})$ 
have bounded first and second order derivatives,
let $(\Omega,\mathcal{F},\P, (\F_t)_{t\in[0,T]})$ be a filtered probability space which satisfies the usual conditions,
let $W=(W_t)_{t\in[0,T]}\colon [0,T]\times\Omega\to\R^{d}$ be a standard
$(\F_t)_{t\in[0,T]}$-Brownian motion with continuous sample paths, and
for every $n\in \N$, $x\in\R^d$ let $Y^{n,x}=(Y^{n,x}_t)_{t\in [0,T]} \colon [0,T]\times\Omega\to\R^d$ satisfy for all
$k\in\{0,1,\ldots,n-1\}$,
 $t\in (\frac{kT}{n},\frac{(k+1)T}{n}]$
that
 $Y^{n,x}_0=x$ and
$Y^{n,x}_t=  Y^{n,x}_{\frac{kT}{n}}+ \mu( Y^{n,x}_{\frac{kT}{n}} ) (t-\frac{kT}{n}) 
+\sigma( Y^{n,x}_{\frac{kT}{n}} ) (W_{t}-W_{\frac{kT}{n}}) 
$. 
Then
\begin{enumerate}[i)]
\item \label{c10a}for every $x\in\R^d$ there exists a unique adapted stochastic process with continuous sample paths
 $X^{x}=(X^{x}_t)_{t\in[0,T]} \colon [0,T]\times\Omega\to\R^d $  
  such that
for all $t\in[0,T]$ it holds a.s.\ that $X_t^{x}=x+\int_{0}^{t}\mu(X_s^{x})\,ds+\int_{0}^{t}\sigma(X_s^{x})\,dW_s$ and
\item \label{c10b}
there exists $C\in (0,\infty)$ such that for all
$n\in\N$,
$x\in\R^d$, $y\in\R^d\setminus\{x\}$, $t\in[0,T]$ it holds that
\begin{align}\begin{split}
&\frac{\left(\E\!\left[\left\lVert X^{x}_t-X^{y}_t\right\rVert^{p}\right]\right)^{\nicefrac{1}{p}}}{\lVert x-y\rVert}+
\frac{\left(
\E\!\left[
\left\lVert Y^{n,x}_{t}-Y^{n,y}_{t} \right\rVert^{p}\right]\right)^{\nicefrac{1}{p}}}{\lVert x-y\rVert}\\
&+
\frac{\sqrt{n} \left(\E \!\left[\bigl\lVert X_{t}^x- Y_{t}^{n,x}\bigr\rVert^p\right]\right)^{1/p}}{1+\lVert x\rVert}
+
\frac
{\sqrt{n} \left(\E \!\left[\bigl\lVert 
\bigl( X_{t}^x- Y_{t}^{n,x}\bigr)
-\bigl( X_{t}^y- Y_{t}^{n,y}\bigr)
\bigr\rVert^p\right]\right)^{1/p}}{\lVert x-y\rVert\bigl(1+\lVert x\rVert+\lVert y\rVert\bigr)}\leq C.
\end{split}\end{align}

\end{enumerate}
\end{corollary}
\begin{proof} [Proof of \cref{c10}]By Jensen's inequality we can assume $p\geq 4$.
Next, the fact that
$\mu\in C^2(\R^d,\R^d) $, $ \sigma  \in C^2(\R^d,\R^{d\times d})$ 
have bounded first and second order derivatives and \cref{f01} show that there exist $b,c\in (0,\infty)$ such that for all $x,y,\tilde{x},\tilde{y}\in\R^d$ it holds that
\begin{equation}
\begin{split}
&
\max _{\zeta\in\{\mu,\sigma\}}
\left\lVert 
(\zeta(x)-\zeta (y))-
(\zeta(\tilde{x})-\zeta (\tilde{y}))\right\rVert 
\leq c
\left\lVert 
(x-y)-(\tilde{x}-\tilde{y})\right\rVert 
+b
\tfrac{\left(\left\lVert x-y\right\rVert 
+\left\lVert \tilde{x}-\tilde{y}\right\rVert 
\right)}{2}\lVert x-\tilde{x}\rVert.
\end{split}\label{a10}
\end{equation}
Throughout the rest of this proof let $V\colon\R^d\to [1,\infty)$ satisfy for all $x\in\R^d$ that
\begin{align}
V(x)= 2^p\Bigl(1+\bigl(\lVert \mu(0) \rVert+ \lVert\sigma(0)\rVert\bigr)^2+ c^2\lVert x \rVert^2\Bigr)^{p/2}
,
\end{align}
 let
$\delta_n\colon [0,T]\to[0,T]$, $n\in\N$, satisfy for all $n\in\N$ that 
$\delta_n([0,\frac{T}{n}])=\{0\}, 
\delta_n((\frac{T}{n},\frac{2T}{n}])=\{\frac{T}{n}\},
\ldots, \delta_n((\frac{(n-1)T}{n},\frac{nT}{n}])=\{\frac{(n-1)T}{n}\}
$,  and 
for every $n\in\N$,
 $s\in [0,T]$, $x\in\R^d$ 
 let $(\mathbb{Y}^{\delta_n,x}_{s,t})_{t\in[s,T]} \colon[s,T]  \times\Omega\to\R^d $ satisfy for all $t\in[s,T]$ that 
$\mathbb{Y}^{\delta_n,x}_{s,s}=x$ and
$
\mathbb{Y}^{\delta_n,x}_{s,t}=
 \mathbb{Y}_{s,\max\{s,\rdown{t}{\delta_n}\}}^{\delta_n,x}+\mu( \mathbb{Y}_{s,\max\{s,\rdown{t}{\delta_n}\}}^{\delta_n,x})(
t-
\max\{s,\rdown{t}{\delta_n}\}
)
+\sigma( \mathbb{Y}_{s,\max\{s,\rdown{t}{\delta_n}\}}^{\delta_n,x})(W_{t}-
W_{\max\{s,\rdown{t}{\delta_n}\}}).
$
Then 
\begin{enumerate}[(A)]\itemsep0pt
\item   for all $x\in\R^d$ it holds that
\begin{align}\label{a11}
\lVert \mu(0) \rVert+ \lVert\sigma(0)\rVert+c\lVert x\rVert\leq 
2 \Bigl(1+\bigl(\lVert \mu(0) \rVert+ \lVert\sigma(0)\rVert\bigr)^2+ c^2\lVert x \rVert^2\Bigr)^{1/2}= (V(x))^{1/p},
\end{align}
\item 
for all $n\in\N$,
 $s\in [0,T]$, $x\in\R^d$ it holds that $(\mathbb{Y}^{\delta_n,x}_{s,t})_{t\in[s,T]}$ has continuous sample paths, and
 \item 
for all $n\in\N$,
 $s\in [0,T]$, $x\in\R^d$, $t\in[s,T]$ 
it holds 
 a.s.\ that
\begin{equation}
\label{y01b}
\mathbb{Y}^{\delta_n,x}_{0,t}=Y_{t}^{n,x}\quad\text{and}\quad \mathbb{Y}^{\delta_n,x}_{s,t}=
x+\int_{s}^{t}\mu( \mathbb{Y}_{s,\max\{s,\rdown{r}{\delta_n}\}}^{\delta_n,x})\,d{r}
+\int_{s}^{t}\sigma( \mathbb{Y}_{s,\max\{s,\rdown{r}{\delta_n}\}}^{\delta_n,x})\,dW_{r}.
\end{equation}

\end{enumerate}
Next, a standard result on stochastic differential equations with Lipschitz continuous coefficients
(see, e.g., \cite[Theorem~V.13.1 and Lemma~V.13.6]{RogersWilliams2000b}) and the fact that
$\mu, \sigma  $ are Lipschitz continuous show that 
for every 
 $s\in [0,T]$, $x\in\R^d$ there exists a unique adapted stochastic process with continuous sample paths $(\mathbb{X}^{x}_{s,t})_{t\in[s,T]} \colon[s,T]  \times\Omega\to\R^d $   such that
for all $t\in[s,T]$ it holds a.s.\ that
\begin{equation}
\label{y01c}
 \mathbb{X}^{x}_{s,t}=
x+\int_{s}^{t}\mu( \mathbb{X}_{s,t}^{x})\,d{r}
+\int_{s}^{t}\sigma( \mathbb{X}_{s,t}^{x})\,dW_{r}
.
\end{equation}
For every $x\in\R^d$ let   $X^{x}=(X^{x}_t)_{t\in[0,T]} \colon [0,T]\times\Omega\to\R^d $   satisfy 
   that $X^x=(\mathbb{X}_{0,t}^x)_{t\in[0,T]}$. 
This and \eqref{y01c} prove
 \eqref{c10a}.
 
 Moreover, the fact that 
$p\geq 3$
 and 
\cref{s28} (applied with
$p\gets p/2$, $a\gets 4 \bigl[1+\bigl(\lVert \mu(0) \rVert+ \lVert\sigma(0)\rVert\bigr)^2\bigr]$, $c\gets 2c$
in the notation of \cref{s28})
show for all $x,y\in\R^d$ that
$
\bigl\lvert((\totalD V)(x))(y)\bigr\rvert
\leq {2p}{c} (V(x))^{\frac{p-1}{p}}\lVert y\rVert$ and $
\bigl\lvert ((\totalD^2 V)(x))(y,y)\bigr\rvert
\leq 4p^2c^2(V (x))^{\frac{p-2}{p}}\lVert y\rVert^2.
$
This, the fact that $p\geq 4$, \eqref{a10}--\eqref{y01c}, 
and
\cref{s01b} (applied for every $n\in\N $ with $\barC\gets \max\{2pc,4p^2c^2\} $,
$\delta\gets \delta_n$, 
$(X^{\funcPi,x}_{s,t} )_{s\in[0,T],t\in[s,T],x\in\R^d}\gets(\mathbb{X}_{s,t}^x)_{s\in[0,T],t\in[s,T],x\in\R^d}$,
$(X^{\delta,x}_{s,t} )_{s\in[0,T],t\in[s,T],x\in\R^d}\gets (\mathbb{Y}_{s,t}^{\delta_n,x})_{s\in[0,T],t\in[s,T],x\in\R^d}$
in the notation of \cref{s01b})
complete the proof of \cref{c10}.
\end{proof}

\begin{corollary}\label{c01}
Let $d\in\N$, $T,p\in (0,\infty)$, 
let $\lVert\cdot\rVert\colon \R^d\to [0,\infty)$ be a norm,
let
$f\in C^2(\R^d,\R)$,
$\mu\in C^2(\R^d,\R^d) $, 
$\sigma\in C^2(\R^d,\R^{d\times d}) $ have bounded first and second order derivatives,
let $(\Omega,\mathcal{F},\P, (\F_t)_{t\in[0,T]})$ be a filtered probability space which satisfies the usual conditions,
 let $W^i\colon [0,T]\times\Omega\to\R^d $, $i\in\N$, be independent standard 
$(\F_t)_{t\in[0,T]}$-Brownian motions with continuous sample paths,
for every $n,i\in \N$, $x\in\R^d$ let $X^{n,i,x}=(X^{n,i,x}_t)_{t\in [0,T]} \colon [0,T]\times\Omega\to\R^d$ satisfy for all
$k\in\{0,1,\ldots,n-1\}$,
 $t\in ({kT}/{n},{(k+1)T}/{n}]$
that
 $X^{n,i,x}_0=x$ and
$X^{n,i,x}_t=  X^{n,i,x}_{{kT}/{n}}+ \mu( X^{n,i,x}_{{kT}/{n}} ) (t-\frac{kT}{n}) 
+\sigma( X^{n,i,x}_{{kT}/{n}} ) (W^i_{t}-W^i_{{kT}/{n}}) 
$, and for every $x\in\R^d$ let $X^{x}=(X^{x}_t)_{t\in[0,T]} \colon [0,T]\times\Omega\to\R^d $  
be an adapted stochastic process with continuous sample paths such that
for all $t\in[0,T]$ it holds a.s.\ that $X_t^{x}=x+\int_{0}^{t}\mu(X_s^{x})\,ds+\int_{0}^{t}\sigma(X_s^{x})\,dW_s^1$.
Then
\begin{align}\footnotesize\begin{split}
\sup_{\substack{x,y\in\R^d,n\in\N\colon\\ x\neq y}}
\frac{\sqrt{n}\left(\E\! 
\left[\left\lvert
\frac{1}{n}\sum\limits_{i=1}^{n} \left[\Bigl(f(X_T^{n,i,x})-\E [f(X_T^{x})]\Bigr)
-\Bigl(f(X_T^{n,i,y})-\E [f(X_T^{y})]\Bigr)\right]\right\rvert^p\right]\right)^{\nicefrac{1}{p}}}{\lVert x-y\rVert(1+\lVert x\rVert+\lVert y\rVert )}
<\infty.
\end{split}\end{align}
\end{corollary}

\begin{proof}[Proof of \cref{c01}]First, observe that by Jensen's inequality we can assume that 
$ p\geq 2$. Next, the assumptions on $f$ and \cref{f01} imply that there exists $c\in (0,\infty)$ such that for all
$x,y,\tilde{x},\tilde{y}\in\R^d$ it holds that
$\lvert f(x)-f(y)\rvert\leq c\lVert x-y\rVert$ and
\begin{align}\begin{split}
&
\lvert
(f(x)-f(y)) -(f(\tilde{x})-f(\tilde{y}))\rvert\leq c
\left[
 \lVert(x-y) -(\tilde{x}-\tilde{y})\rVert+ \tfrac{\lVert x-y\rVert+\lVert \tilde{x}-\tilde{y}\rVert}{2}\lVert x-\tilde{x}\rVert\right].
\end{split}\label{c03b}\end{align}
This, the triangle inequality, H\"older's inequality,  \cref{c10}, and the assumptions on $\mu,\sigma$ imply that
\begin{align}\begin{split}
&\sup_{\substack{x,y\in\R^d,n\in\N\colon x\neq y}}
\frac{\left(\E\!\left[\left\lvert f(X^{n,1,x}_{T})-f(X^{n,1,y}_{T}) \right\rvert^{p}\right]\right)^{\nicefrac{1}{p}}}{\lVert x-y\rVert }\\
&\leq 
\sup_{\substack{x,y\in\R^d,n\in\N\colon x\neq y}}
\frac{c\left(
\E\!\left[
\left\lVert X^{n,1,x}_{T}-X^{n,1,y}_{T} \right\rVert^{p}\right]\right)^{\nicefrac{1}{p}}}{\lVert x-y\rVert}<\infty
\end{split}\label{c03}
\end{align}
and
\begin{align}\begin{split}
&\sup_{\substack{x,y\in\R^d,n\in\N\colon\\ x\neq y}}\frac{\sqrt{n} \left(\E\!\left[\left\lvert \left(f(X_T^{n,1,x}) -f(X^{x}_T)\right)
-\left(f(X_T^{n,1,y}) -f(X^{x}_T)\right)\right\rvert^{p}\right]\right)^{\nicefrac{1}{p}}}{\lVert x-y\rVert(1+\lVert x\rVert+\lVert y\rVert )}
\\
&\leq c\sup_{\substack{x,y\in\R^d,n\in\N\colon\\ x\neq y}}\Biggl(\E \Biggl[\biggl\lvert   \frac{\sqrt{n}
\left\lVert \left(X_T^{n,1,x} -X^{x}_T\right)
-\left(X_T^{n,1,y} - X^{y}_T\right)\right\rVert
}{(1+\lVert x\rVert+\lVert y\rVert )\lVert x-y\rVert}\\
&\qquad\qquad\qquad\qquad\qquad
+\frac{0.5\sqrt{n}\Bigl[
\left\lVert X_T^{n,1,x} -X^{x}_T\right\rVert
+\left\lVert X_T^{n,1,y} -X^{y}_T\right\rVert\Bigr] \bigl\lVert X^{x}_T-X^{y}_T\bigr\rVert}{(1+\lVert x\rVert+\lVert y\rVert )\lVert x-y\rVert}\biggr\rvert^p\Biggr]\Biggr)^{\nicefrac{1}{p}}\\
\end{split}\nonumber\\
\begin{split}
&\leq c\sup_{\substack{x,y\in\R^d,n\in\N\colon\\ x\neq y}}  \frac{\sqrt{n}
\left(
\E\!\left[
\left\lVert \left(X_T^{n,1,x} -X^{x}_T\right)
-\left(X_T^{n,1,y} - X^{y}_T\right)\right\rVert^p\right]\right)^{\!\nicefrac{1}{p}}
}{(1+\lVert x\rVert+\lVert y\rVert )\lVert x-y\rVert}\\
& 
+c\left[\sup_{x\in\R^d,n\in\N} \frac{\sqrt{n}
\left(\E\!\left[
\left\lVert X_T^{n,1,x} -X^{x}_T\right\rVert^{2p}\right]\right)^{\frac{1}{2p}}
 }{1+\lVert x\rVert}\right]\!\!\!
\left[
\sup_{\substack{x,y\in\R^d\colon x\neq y}}
\frac{\left(\E\!\left[\left\lVert X^{x}_T-X^{y}_T\right\rVert^{2p}\right]\right)^{\frac{1}{2p}}}{\lVert x-y\rVert}\right]<\infty.
\label{c04}\end{split}\end{align}
Next, for all $x,y\in\R^d$, $n\in\N$ it holds that
\begin{align}\begin{split}
&
\frac{1}{n}\sum_{i=1}^{n} 
\Bigl(f(X_T^{n,i,x})-\E [f(X_T^{x})]\Bigr)
-\Bigl(f(X_T^{n,i,y})-\E [f(X_T^{y})]\Bigr)
\\
&
= 
\frac{1}{n}\sum_{i=1}^{n} 
\Bigl(f(X_T^{n,i,x})-f(X_T^{n,i,y}) \Bigr)
-\Bigl(\E\!\left[f(X_T^{n,i,x})-f(X_T^{n,i,y})\right]\Bigr)
\\
&\quad 
+  \E\!\left[ \left(f(X_T^{n,1,x})-f(X_T^{x})\right)-\left(f(X_T^{n,1,y})-f(X_T^{y})\right) \right]  
.
\end{split}\label{c02}\end{align}
Furthermore,
the Marcinkiewicz-Zygmund inequality (see \cite[Theorem~2.1]{Rio09}), the fact that $\exponentLP\in [2,\infty)$,  the triangle inequality, and  Jensen's inequality show that for all $n\in\N$ and all i.i.d.\ 
 integrable random variables $\mathfrak{X}_1,\mathfrak{X}_2,\ldots,\mathfrak{X}_n\colon\Omega\to\R$ it holds that
$
\left(\E\!\left[ \left\lvert \sum_{k=1}^{n}(\mathfrak{X}_k-\E[\mathfrak{X}_k])\right\rvert^p\right]\right)^{\nicefrac{1}{p}}\leq \sqrt{n}\sqrt{\exponentLP-1}\bigl(\E[\lvert\mathfrak{X}_1-\E[\mathfrak{X}_1]\rvert^p]\bigr)^{\nicefrac{1}{p}}
\leq 2\sqrt{n}\sqrt{\exponentLP-1}\bigl(\E[\lvert\mathfrak{X}_1\rvert^p]\bigr)^{\nicefrac{1}{p}}.
$ This, \eqref{c02}, the independence assumptions, the triangle inequality,
\eqref{c03}, and \eqref{c04} show that
\begin{align}
&\sup_{\substack{x,y\in\R^d,n\in\N\colon\\ x\neq y}}
\left[\sqrt{n}
\left(
\E\!\left[\left\lvert
\frac{1}{n}\sum_{i=1}^{n} 
\frac{\Bigl(f(X_T^{n,i,x})-\E [f(X_T^{x})]\Bigr)
-\Bigl(f(X_T^{n,i,y})-\E [f(X_T^{y})]\Bigr)}{\lVert x-y\rVert(1+\lVert x\rVert+\lVert y\rVert )}\right\rvert^p\right]\right)^{\nicefrac{1}{p}}\right]\nonumber
\\
&
\leq \sup_{\substack{x,y\in\R^d,n\in\N\colon\\ x\neq y}}\left[2\sqrt{p-1}\left(\E\!\left[\left\lvert
 \frac{
f(X_T^{n,1,x})
-f(X_T^{n,1,y})}{\lVert x-y\rVert(1+\lVert x\rVert+\lVert y\rVert )}\right\rvert^p\right]\right)^{\nicefrac{1}{p}}\right]\nonumber
\\
&\quad 
+\sup_{\substack{x,y\in\R^d,n\in\N\colon\\ x\neq y}} \frac{\sqrt{n}\left(\E\!\left[\left\lvert \left(f(X_T^{n,1,x})-f(X_T^{x})\right)-\left(f(X_T^{n,1,y})-f(X_T^{y})\right) \right\rvert^p\right]\right)^{\nicefrac{1}{p}}}{\lVert x-y\rVert(1+\lVert x\rVert+\lVert y\rVert )}<\infty
.
\end{align}
This completes the proof of \cref{c01}.
\end{proof}

\paragraph*{Acknowledgement}
This work has been
funded by the Deutsche Forschungsgemeinschaft (DFG, German Research Foundation) through
the research grant HU1889/7-2.

{\small

\bibliographystyle{acm}
\bibliography{bibfile}

\def\cprime{$'$}
\begin{thebibliography}{10}

\bibitem{BaoHuangZhang2020}
{\sc Bao, J., Huang, X., and Zhang, S.-Q.}
\newblock Convergence rate of {EM} algorithm for {SDEs} under integrability
  condition.
\newblock {\em arXiv preprint arXiv:2009.04781\/} (2020).

\bibitem{CoxHutzenthalerJentzen2014}
{\sc Cox, S.~G., Hutzenthaler, M., and Jentzen, A.}
\newblock Local {L}ipschitz continuity in the initial value and strong
  completeness for nonlinear stochastic differential equations.
\newblock {\em arXiv:1309.5595v2\/} (2014), 1--84.

\bibitem{dz92}
{\sc {Da Prato}, G., and Zabczyk, J.}
\newblock {\em Stochastic equations in infinite dimensions}, vol.~44 of {\em
  Encyclopedia of Mathematics and its Applications}.
\newblock Cambridge University Press, Cambridge, 1992.

\bibitem{DareiotisGerencserLe2021}
{\sc Dareiotis, K., Gerencs{\'e}r, M., and L{\^e}, K.}
\newblock Quantifying a convergence theorem of {Gy\"ongy} and {K}rylov.
\newblock {\em arXiv preprint arXiv:2101.12185\/} (2021).

\bibitem{EHutzenthalerJentzenKruse2016}
{\sc E, W., Hutzenthaler, M., Jentzen, A., and Kruse, T.}
\newblock Multilevel {P}icard iterations for solving smooth semilinear
  parabolic heat equations.
\newblock {\em arXiv:1607.03295\/} (2016).
\newblock Springer Nature Partial Differential Equations and Applications (in
  press).

\bibitem{GrohsWurstemberger2018}
{\sc {Grohs}, P., {Hornung}, F., {Jentzen}, A., and {von Wurstemberger}, P.}
\newblock {A proof that artificial neural networks overcome the curse of
  dimensionality in the numerical approximation of Black-Scholes partial
  differential equations}.
\newblock {\em to appear in Mem. Amer. Math. Soc.\/} (2019).

\bibitem{h98}
{\sc Heinrich, S.}
\newblock Monte {C}arlo complexity of global solution of integral equations.
\newblock {\em J. Complexity 14}, 2 (1998), 151--175.

\bibitem{h01}
{\sc Heinrich, S.}
\newblock Multilevel {M}onte {C}arlo methods.
\newblock In {\em Large-Scale Scientific Computing}, vol.~2179 of {\em Lect.
  Notes Comput. Sci.} Springer, Berlin, 2001, pp.~58--67.

\bibitem{hudde2021stochastic}
{\sc Hudde, A., Hutzenthaler, M., and Mazzonetto, S.}
\newblock {A stochastic Gronwall inequality and applications to moments, strong
  completeness, strong local Lipschitz continuity, and perturbations}.
\newblock In {\em Annales de l'Institut Henri Poincar{\'e}, Probabilit{\'e}s et
  Statistiques\/} (2021), vol.~57, Institut Henri Poincar{\'e}, pp.~603--626.

\bibitem{HutzenthalerJentzen2014Memoires}
{\sc Hutzenthaler, M., and Jentzen, A.}
\newblock Numerical approximations of stochastic differential equations with
  non-globally {L}ipschitz continuous coefficients.
\newblock {\em Mem. Amer. Math. Soc. 4\/} (2015), 1--112.

\bibitem{HJ20}
{\sc Hutzenthaler, M., and Jentzen, A.}
\newblock {On a perturbation theory and on strong convergence rates for
  stochastic ordinary and partial differential equations with non-globally
  monotone coefficients}.
\newblock {\em {Annals of Probability} 48}, 1 (2020), 53--93.

\bibitem{hjk11}
{\sc Hutzenthaler, M., Jentzen, A., and Kloeden, P.~E.}
\newblock Strong and weak divergence in finite time of {E}uler's method for
  stochastic differential equations with non-globally {L}ipschitz continuous
  coefficients.
\newblock {\em Proc. R. Soc. Lond. Ser. A Math. Phys. Eng. Sci. 467\/} (2011),
  1563--1576.

\bibitem{HutzenthalerJentzenKloeden2012}
{\sc Hutzenthaler, M., Jentzen, A., and Kloeden, P.~E.}
\newblock Strong convergence of an explicit numerical method for {SDE}s with
  non-globally {L}ipschitz continuous coefficients.
\newblock {\em Ann. Appl. Probab. 22}, 4 (2012), 1611--1641.

\bibitem{HutzenthalerJentzenKloeden2013}
{\sc Hutzenthaler, M., Jentzen, A., and Kloeden, P.~E.}
\newblock Divergence of the multilevel {M}onte {C}arlo {E}uler method for
  nonlinear stochastic differential equations.
\newblock {\em Ann. Appl. Probab. 23}, 5 (2013), 1913--1966.

\bibitem{hjk2021overcoming}
{\sc Hutzenthaler, M., Jentzen, A., and Kruse, T.}
\newblock Overcoming the curse of dimensionality in the numerical approximation
  of parabolic partial differential equations with gradient-dependent
  nonlinearities.
\newblock {\em Foundations of Computational Mathematics\/} (2021), 1--62.

\bibitem{hjkn2021BSDE}
{\sc Hutzenthaler, M., Jentzen, A., Kruse, T., and Nguyen, T.~A.}
\newblock Overcoming the curse of dimensionality in the numerical approximation
  of backward stochastic differential equations.
\newblock {\em arXiv preprint arXiv:2108.10602\/} (2021).

\bibitem{HJK+18}
{\sc Hutzenthaler, M., Jentzen, A., Kruse, T., Nguyen, T.~A., and von
  Wurstemberger, P.}
\newblock Overcoming the curse of dimensionality in the numerical approximation
  of semilinear parabolic partial differential equations.
\newblock {\em Proceeding of the Royal Society A 476}, 20190630 (2020).

\bibitem{KP95}
{\sc Kloeden, P.~E., and Platen, E.}
\newblock {\em Numerical solution of stochastic differential equations},
  vol.~23 of {\em Applications of Mathematics}.
\newblock Springer-Verlag, 1999.

\bibitem{LeLing2021}
{\sc L{\^e}, K., and Ling, C.}
\newblock Taming singular stochastic differential equations: {A} numerical
  method.
\newblock {\em arXiv preprint arXiv:2110.01343\/} (2021).

\bibitem{LeobacherSzoelgyenyi2016}
{\sc Leobacher, G., and Sz{\"o}lgyenyi, M.}
\newblock A numerical method for {SDEs} with discontinuous drift.
\newblock {\em BIT Numerical Mathematics 56}, 1 (2016), 151--162.

\bibitem{Mao07}
{\sc Mao, X.}
\newblock {\em Stochastic differential equations and their applications}.
\newblock Horwood Publishing Limited, Chichester, 2007.

\bibitem{NeuenkirchSzoelgyenyi2021}
{\sc Neuenkirch, A., and Sz{\"o}lgyenyi, M.}
\newblock The {E}uler--{M}aruyama scheme for {SDEs} with irregular drift:
  convergence rates via reduction to a quadrature problem.
\newblock {\em IMA Journal of Numerical Analysis 41}, 2 (2021), 1164--1196.

\bibitem{NeuenkirchEtAl2019}
{\sc Neuenkirch, A., Sz\"olgyenyi, M., and Szpruch, L.}
\newblock An adaptive {E}uler--{M}aruyama scheme for stochastic differential
  equations with discontinuous drift and its convergence analysis.
\newblock {\em SIAM Journal on Numerical Analysis 57}, 1 (2019), 378--403.

\bibitem{NgoTaguchi2017}
{\sc Ngo, H.-L., and Taguchi, D.}
\newblock On the {E}uler--{M}aruyama approximation for one-dimensional
  stochastic differential equations with irregular coefficients.
\newblock {\em IMA Journal of Numerical Analysis 37}, 4 (2017), 1864--1883.

\bibitem{Rio09}
{\sc Rio, E.}
\newblock Moment inequalities for sums of dependent random variables under
  projective conditions.
\newblock {\em Journal of Theoretical Probability 22\/} (2009), 146–--163.

\bibitem{RogersWilliams2000b}
{\sc Rogers, L. C.~G., and Williams, D.}
\newblock {\em Diffusions, {M}arkov processes and martingales. {V}ol. 2}.
\newblock Cambridge Mathematical Library. Cambridge University Press,
  Cambridge, 2000.
\newblock It{\^o} calculus, Reprint of the second (1994) edition.

\bibitem{Sabanis2013ECP}
{\sc Sabanis, S.}
\newblock A note on tamed {E}uler approximations.
\newblock {\em Electron. Commun. Probab. 18\/} (2013), 1--10.

\bibitem{Sabanis2016}
{\sc Sabanis, S.}
\newblock Euler approximations with varying coefficients: the case of
  superlinearly growing diffusion coefficients.
\newblock {\em The Annals of Applied Probability 26}, 4 (2016), 2083--2105.

\end{thebibliography}
}

\end{document}